\documentclass[11pt]{article}
\usepackage{a4wide}
\usepackage{color,epsfig}
\usepackage{amsmath}
\usepackage{amsfonts}
\usepackage{dsfont}
\usepackage{mathrsfs}
\usepackage{amssymb}
\usepackage{amsfonts,graphicx,psfrag}
\usepackage{cite}
\usepackage{amsthm}
\usepackage{hyperref}
\usepackage{enumerate}
\usepackage{mathtools}
\usepackage{enumitem}
\mathtoolsset{showonlyrefs}

\newtheorem{Th}{Theorem}[section]

\newtheorem{lemma}{Lemma}[section]
\newtheorem{theorem}[lemma]{Theorem}
\newtheorem{corollary}[lemma]{Corollary}

\newtheorem{proposition}[lemma]{Proposition}
\newtheorem{definition}[lemma]{Definition}

\newtheorem{remark}[lemma]{Remark}
\let\lutzremark=\remark
\def\remark{\lutzremark\normalfont}

\newtheorem{notation}[lemma]{Notation}

\def\be{\begin{equation}}
\def\ee{\end{equation}}
\def\bea{\begin{eqnarray}}
\def\eea{\end{eqnarray}}
\def\bes{\begin{eqnarray*}}
\def\ees{\end{eqnarray*}}

\def\nn{\nonumber}
\def\<{\langle}
\def\>{\rangle}
\def\lb{\label}
\def\bs{\setminus}
\def\pt{\partial}
\def\d{{\mathrm{d}}}

\def\R{{\bf R}}
\def\C{{\bf C}}
\def\Z{{\bf Z}}
\def\N{{\bf N}}
\def\U{{\bf U}}

\def\aa{{\alpha}}
\def\bb{{\beta}}
\def\ga{{\gamma}}

\def\th{{\theta}}

\def\om{{\omega}}

\def\ep{{\epsilon}}
\def\lm{{\lambda}}

\def\sg{{\sigma}}

\def\cA{{\cal A}}

\def\P{{\cal P}}

\def\diag{{\rm diag}}

\def\Sp{{\rm Sp}}

\def\dm{{\rm \diamond}}

\def\ol#1{\overline{#1}}
\def\td#1{\tilde{#1}}

\title{Linear stability of the elliptic relative equilibria for the restricted $4$-body problem: the Euler case}

\author{Bowen Liu$^{1}$\thanks{Partially supported by NSFC (Nos. {12101394,  12171426}), Science and Technology Innovation Action Program of  STCSM (No. 20JC1413200), Natural Science Foundation of Shanghai No. 22ZR1433100	and Innovation Program of Shanghai Municipal Education Commission. E-mail: liubowen2010@gmail.com} \quad and \quad
	Qinglong Zhou$^{2}$\thanks{Corresponding author. Partially supported by NSFC (No.12171426), the Natural Science Foundation of Zhejiang Province (No. Y19A010072) and the Fundamental Research Funds for the Central Universities (No. 2021FZZX001-01).
		E-mail: zhouqinglong@zju.edu.cn. }\\
	$^{1}$ School of Mathematical Science,\\ Shanghai Jiao Tong University, Shanghai 200240, China\\
	$^{2}$ Department of Mathematics,\\Zhejiang University, Hangzhou 310027, Zhejiang, China\\
}

\begin{document}

\maketitle

\begin{abstract}
In this paper, we consider the elliptic relative equilibria of the restricted $4$-body problems, where
the three primaries form an Euler collinear configuration and the four bodies span $\R^2$. We obtain the symplectic reduction to the general restricted $N$-body problem. By analyzing the relationship between this restricted $4$-body problems and the elliptic Lagrangian solutions, we obtain the linear stability of the restricted $4$-body problem by the $\om$-Maslov index. Via numerical computations, we also obtain conditions of the stability on the mass parameters for the symmetric cases.
\end{abstract}



\renewcommand{\theequation}{\thesection.\arabic{equation}}

\setcounter{equation}{0}
\section{Introduction and main results}
\label{sec:1}

{In the classical planar $N$-body problems of celestial mechanics, the position vectors of the $N$-particles are denoted by $q_1 ,\dots , {q_N}\in \R^2$, and the masses are represented by
$m_1 ,\dots,{m_N} > 0$.} By Newton’s second law and the law of universal
gravitation, the system of equations is
\begin{align}
m_i\ddot{q_i} = \frac{\pt U}{\pt q_i}, \quad i = 1, {\dots, N},\lb{1.1}
\end{align}
where $U(q) = U(q_1, {\dots, q_{N}}) = \sum_{1\leq i< j\leq {N}} \frac{m_i m_j}{|q_i-q_j|}$ 
is the potential function and $|\cdot |$ {is} the standard norm of vector in $\R^2$. 
Suppose the configuration space is 
$$\hat{\chi} :=\left\{q =(q_1, {\dots, q_N})\in (\R^2)^{{N}}\l|\sum_{i = 1}^{{N}} m_iq_i = 0, q_i \neq q_j, \forall i\neq j \right\}. $$
For the period $T$, the corresponding action functional is 
\begin{align}
  \mathbf{A}(q) = \int_{0}^{T} \left[ \sum_{i = 1}^{{N}} \frac{m_i|\dot{q}_i(t)|^2}{2} +U(q(t))\right] \d t, \lb{1.222}
  \end{align}
which is defined on the loop space $W^{1,2} (\R/T \Z, \hat{\chi})$. The periodic solutions
of \eqref{1.1} correspond to critical points of the action functional \eqref{1.222}. 
Let $p_1 , {\dots, p_{N}} \in \R^2$ be the momentum vectors of the particles respectively. It is {well-known} that \eqref{1.1} can be reformulated as a Hamiltonian system by 
\begin{align}
\dot{p}_i = -\frac{\pt H}{\pt q_i}, \; \dot{q}_i  = \frac{\pt H}{\pt p_i}, \quad \mbox{for} \;i = 1, {\dots, N}, \lb{1.2}
\end{align}
with the Hamiltonian function
\begin{align} 
H(p, q) = \sum_{i =1}^{{N}} \frac{|p_i|^2}{2m_i} - U(q_1,{\dots, q_{N}}).\lb{1.3}
\end{align}

One special class of periodic solutions to the planar $N$-body problem is the elliptic relative equilibrium (ERE for short) \cite{MS}. It is generated by a central configuration and the Keplerian motion.
A central configuration (C.C. for short) is formed by $N$ position vectors $\left(q_{1}, \ldots, q_{N}\right)=\left(a_{1}, \ldots, a_{N}\right)$ which satisfy
\begin{align}
	-\lambda m_{i} q_{i}=\frac{\partial U}{\partial q_{i}}, \forall\; 1\leq i \leq N, \lb{eqn:cc}
\end{align}
where $\lambda =U(a) /  {I(a)}>0$ and $I(a)={\sum_{i=1}^N} m_{i}|a_{i}|^{2}$ is the moment of inertia.
A planar central configuration of the $N$-body problem gives rise to a solution of \eqref{1.1} where each particle moves on a specific Keplerian orbit while the totality of the particles move according to a homothety motion.

The restricted 4-body problem is one special case of the general $4$-body problem which are three primaries and one massless body. We assume that $m_1 + m_2 + m_3 = 1$ and $m_4 = 0$. By $m_4 = 0$, \eqref{eqn:cc} can be reduced to
\begin{align}\lb{eqn:N-1CC}
    -\lambda q_{i}=& \sum_{1\leq j \leq 3,
    i\neq j}\frac{m_i(q_j - q_i)}{|q_j - q_i|^3}, \quad \mbox{for} \quad  1 \leq i \leq 3\\
    -\lambda q_{4}=& \sum_{1\leq j \leq 3
   }\frac{m_i(q_j - q_4)}{|q_j - q_4|^3}.
\end{align} 
The C.C. is decomposed to a 3-body problem and the actions on the massless body from three primaries.
For the 3-body problem, there exists two types of the C.C.s: the Lagrangian equilateral configurations \cite{Lag} and the Euler collinear configurations \cite{Euler}. If the three primaries form a Lagrangian equilateral, the number of the C.C. of the corresponding restricted 4-body problem has been well-studied \cite{Leandro2003}. For the Euler collinear case, we still have two cases: the $4$ bodies are collinear and the four bodies span $\R^2$. The case of  the four bodies form a collinear configuration has been discussed in \cite{LiZ}.  
In this paper, we focus on the case of $4$ bodies span $\R^2$.

The linear stability of the ERE is determined by the eigenvalues of the linearized Poincar\'e map. 
Let $\U$ denote the unit circle in the complex plane. The ERE is spectrally stable if all eigenvalues of linearized Poincar\'e map  are on $\U$; it is linearly stable if Poincar\'e map is semi-simple and spectrally stable; it is linearly unstable if at least one pair of eigenvalues are not on $\U$.
Since the nineteenth century \cite{R2}, the researches on the stability have always been active in celestial mechanics because it reveals the dynamics near the period orbits. 
However, it has always been one difficult task to obtain the linear stability of ERE, because the linearized Hamiltonian systems are non-autonomous, especially when for the elliptic orbits.
Many results on linear stability of the three-body problems have been obtained over the past decades by numerical methods \cite{MSS,MSS1,MSS2}, bifurcation theory \cite{R1} and the index theory \cite{HS,HLS,HuOuWang2015ARMA,Zhou2017}.  
To the best of our knowledge, the $\om$-Maslov index theory is the only analytical method to obtain the full picture of the stability and instability to the ERE, such as the elliptic Lagrangian solution \cite{HS,HLS,HuOuWang2015ARMA}, and the elliptic Euler solution \cite{Zhou2017,ZhonLong2017CMDA}.

When it comes to four-body problems,
the research on stability of ERE is quite active in the past decades. 
One well-studied case is the elliptic rhombus solution \cite{Mansur2017, Leandro2018, Liu2021} which are linearly instable. 
For the restricted $4$-body problem with three primaries forming Lagrangian equilateral configuration, the full bifurcation diagram of the stability and instability has been obtained for all possible masses and all eccentricity \cite{LiZ, Leandro2003}. Furthermore, one stable region of the linear stability has been found.
Regarding the linear stability of other ERE to the four-body problem, readers may refer to \cite{zhou2019linear, Hu2020}

In this paper, we focus on the restricted 4-body problem with three primaries forming Euller collinear configurations and the four bodies span $\R^2$. 
Since the ERE of the three primaries is always linearly unstable
\cite{Zhou2017},
it is reasonable to study the stability problem of the
massless particle.
For example, the
“massless body” can be imaged as a space station and the three massive bodies form a planetary system.

We first apply the symplectic reduction method \cite{MS} to the general restricted $N$-body problem. 
Let $P$ denote the inertial position of the massless body zero, which moves in the gravitational field of the $N-1$ primaries, without disturbing their motion. The corresponding 
 Hamiltonian function for this $N$th body is given by
\begin{equation}
H(P,q, t)={1\over2}|P|^2-\sum_{i=1}^{N-1}{m_i\over|r(t)R(\th(t))a_i-q|}.\lb{HaFun}
\end{equation}
Here $z=z(t)$ is the Kepler elliptic orbit given through the true anomaly $\th=\th(t)$by
\begin{align}
  r(\th(t)) = |z(t)| = \frac{p}{1+e\cos\th(t)},  \lb{rTh}
\end{align}
where $p=a(1-e^2)$ and $a>0$ is the latus rectum of the ellipse \eqref{rTh}.
By Proposition \ref{P2.1}, the ERE $(P(t),q(t))^T$  of the system \eqref{1.2} is in time $t$ where $q(t)=r(t)R(\theta(t))a_{N}$ and $P(t)=\dot{q}(t)$. It 
can be transformed to the new solution $(\bar{Z}(\theta),\bar{z}(\theta))^T = (0,\sigma,\sigma,0)^T$ in the true anomaly $\theta$ as the new
variable for the original Hamiltonian function $H$ given by \eqref{eqn:red.Ham}. We then have the reduced Hamiltonian system is given in the theorem.

\begin{theorem}\label{linearized.Hamiltonian}
The linearized Hamiltonian system of \eqref{eqn:red.Ham}  at the ERE
  $\zeta_0 \equiv (\bar{Z}(\theta),\bar{z}(\theta))^T =
  (0,\sigma,\sigma,0)^T\in\R^4  $
  depending on the true anomaly $\theta$ is given by
  \begin{align}
    \dot\xi(\theta) = JB(\theta)\xi(\theta),  \lb{eqn:LinearHam1}
  \end{align}
  with
  \bea B(\theta)
  = H''(\theta,\bar{Z},\bar{z})|_{\bar\xi=\xi_0}
  = \left(\begin{array}{cccc|cccc}
  I_2   &-J_2 \\
  J_2     &I_2-{1\over1+e\cos\th}D
  \end{array}\right),  \lb{LinearHam2}\eea
  where
  \begin{align}
    \label{matrix.D}
  D=I_2-{1\over\mu}\left(\sum_{i=1}^{N-1}{m_i\over|a_i-a_{N}|^3}\right)I
  +{3\over\mu}\sum_{i=1}^{N-1} m_i\frac{(a_i-a_{N})(a_i-a_{N})^T}{|a_i-a_{N}|^5}.
  \end{align}
  The corresponding quadratic Hamiltonian function is given by
  \bea
  H_2(\theta,\bar{Z},\bar{z})  
   =\frac{1}{2}|\bar{Z}|^2+\bar{Z}\cdot J\bar{z}
    +\frac{1}{2}\left(I_2-{D\over1+e\cos\th}\right)|\bar{z}|^2. \eea
  \end{theorem}

We apply Theorem \ref{thm:RE.norm.form} to the restricted $4$-body problem. Denote the eigenvalues of $D$ by $\lm_3$ and $\lm_4$. By Proposition  \ref{prop.sum.lm3.lm4}, we have that both $\lm_3$ and $\lm_4$ are positive and  $\lm_3 + \lm_4 = 3$.
By this property, we show the relationship between the Maslov index of $\xi_{\aa, e}$ given by \eqref{eqn:LinearHam1} and the Maslov index of $\ga_{\beta, e}$ which is the elliptic Lagrangian solutions in \cite{HLS}. 
Since the elliptic Lagrangian solutions has already been well-studied in \cite{HLS},
we have the Maslov index of $\xi_{\aa, e}$ in Proposition \ref{prop:index}.
By the properties of the Maslov index, we define three curves $\aa_k(e)$, $\aa_s(e)$ and $\aa_m(e)$ (cf. \eqref{eqn:aa.sm}, \eqref{eqn:def.aak} below) from left to right in $(\aa, e)\in [0,3]\times [0, 1)$ according to the $\omega$-Maslov indices of $\xi_{\aa,e}$. 
By the $\om$-Maslov index theory \cite{Lon4}, we obtain the normal forms of $\xi_{a,e}(2\pi)$ and then the linear stability as follows.
\begin{theorem}\lb{thm:RE.norm.form}
  \begin{enumerate}[label = (\roman*)]
         \item 
         We have $\xi_{\alpha, e}(2 \pi) \approx R\left(\theta_{1}\right) \diamond R\left(\theta_{2}\right)$ for some $\theta_{1}$ and $\theta_{2} \in(\pi, 2 \pi)$, and thus it is strongly linearly stable on the segment $\alpha_{m}(e)<\alpha<3$; 
         \item We have $\left.\xi_{\alpha, e}(2 \pi) \approx D(\lambda) \diamond R(\theta)\right)$ for some $0>\lambda \neq-1$ and $\theta \in(\pi, 2 \pi)$ and it is elliptic-hyperbolic, and thus linearly unstable on the segment $\alpha_{s}(e)<\alpha<\alpha_{m}(e)$.
         \item  We have $\xi_{\alpha, e}(2 \pi) \approx R\left(\theta_{1}\right) \diamond R\left(\theta_{2}\right)$ for some $\theta_{1} \in(0, \pi)$ and $\theta_{2} \in(\pi, 2 \pi)$ with $2 \pi-\theta_{2}<\theta_{1}$, and thus it is strongly linearly stable on the segment $\alpha_{k}(e)<\alpha<\aa_s(e)$.
  \end{enumerate}
\end{theorem}

Note that it is also possible that $ \aa_{k}(e) = \aa_s(e)$, or $\aa_s(e) = \aa_m(e)$ for some $e$ in $[0,1)$. For these cases, the corresponding normal forms of $\xi_{\aa,e}$ are given in Theorem \ref{thm:RE.lim}.

For the circular case, we first numerically relationships between $\aa$ and $m_1$, $m_2$ and $m_3$. Since $m_1 + m_2 + m_3 = 1$, we can plot the stable regions with respect to in $m_1$ and $m_2$ by definition of $\aa$ in \eqref{eqn:ca}, \eqref{eqn:sum.l3.l4} and Corollary \ref{coro:cir.sta}. We show this results in (a) of Figure \ref{fig:2}.
\begin{figure}[htbp]
	\centering
  \begin{tabular}{cc}
    \includegraphics[width=0.4\textwidth]{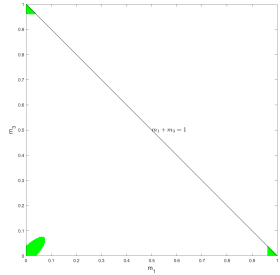} &
		\includegraphics[width=0.4\textwidth]{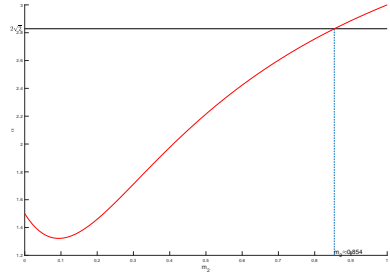}  \\
		\footnotesize{(a) The linear stability of circular solution.}  & 
		\footnotesize{(b) The symmetric case.} 
	\end{tabular}
	\caption{\footnotesize{In Figure (a), the shade regions in the $(m_1,m_3)$-plane show the all the possible choices of $(m_1, m_2, m_3) = (m_1, 1- m_1 - m_3, m_3)$ of the circular orbit such that the system is linear stable. The Figure (b) shows the changes of $\alpha$ with respect to $m_2$ in the symmetric case}}
	\label{fig:2}
\end{figure}

If we further assume that $m_1 = m_3$. 
We have the following numerical results holds shown in (b) of Figure \ref{fig:2}.
\begin{theorem}\lb{thm:numerically}
  If $e = 0$ and $m_1 = m_3$, $\xi_{\aa,e}$ is linear stable for $m_2 \in (0.854,1)$.
\end{theorem}
This paper is organized as follows. We first introduce the generalized symplectic reduction method to the restricted $N$-body problem in Section \ref{sec:2}. We then use the $\om$-Maslo index theory to obtain the linear stability of the restricted $4$-body problem in Section \ref{sec:3}. We consider the symmetric case by assuming $m_1 = m_3$ in Section \ref{sec:3}. We apply the linear stability results in Section \ref{sec:3} and obtain the condition on $m_1$, $m_2$ and $m_3$ such that the system is linear stable.

\setcounter{equation}{0}
\section{Reduction for the restricted $N$-body problem}\label{sec:2}
The central configuration coordinates
for a class of periodic solutions of the $n$-body problem was introduced \cite{MS}. 
In this section, we generalize this reducution to the restricted $N$-body problem.
For the given masses $m=(m_1,m_2,\ldots,m_{N-1})\in (\R^+)^{N-1}$
of $N-1$ primaries
let $a=(a_1,a_2,\ldots,a_{N -1})$ be an $(N - 1)$-body central configuration of $m$  satisfying \eqref{eqn:N-1CC}.
Using normalization $2I(a) = \sum_{i = 1}^{N-1}m_i |a_i|^2 = 1$ and assuming  $\mu = U(a) = \sum_{1\le i<j\le N-1}\frac{m_im_j}{|a_i-a_j|}$, we have
\begin{align}
  \sum_{j=1,j\ne i}^{N-1}\frac{m_j(a_{j}-a_{i})}{|a_{j}-a_{i}|^3}
          = -\frac{U(a)}{2I(a)}a_{i}=-\mu a_{i}.    \label{eq.of.cc}
\end{align}

\begin{proposition}\label{P2.1}
There exists a symplectic coordinate change
$
  \xi = (P,q)^T
    \;\mapsto\; \bar{\xi}
    = ( \bar{Z}, \bar{z})^T,$
such that using the true anomaly $\th$ as the variable the resulting Hamiltonian function of the
$n$-body problem is given by
\begin{align}\lb{eqn:red.Ham}
  H(\theta,\bar{Z},\bar{z})=\frac{1}{2}|\bar{Z}|^2 + \bar{z}\cdot J\bar{Z}
+ \frac{p-r}{2p}|\bar{z}|^2
- \frac{r}{\sg}\sum_{i=1}^{N-1}{m_i\over|\sg a_i-\bar{z}|},
\end{align}
where $J =( 
\begin{smallmatrix}
       0 &-1 \\ 1 & 0
\end{smallmatrix})$, $r(\th)=\frac{p}{1+e\cos\th}$,
$\mu$ defined by \eqref{eq.of.cc}, $\sg=(\mu p)^{1/4}$ and $p$ is given in \eqref{rTh}.
\end{proposition}
\begin{proof}
  Inspired by Lemma 3.1 of \cite{MS}, we carry the coordinate
  changes in four steps.
  
  {\bf Step 1.} {\it Rotating coordinates via the matrix $R(\th(t))$ in time $t$.}
  
  We change first the coordinates $\xi$ to $\hat{\xi}=(\hat{Z}, \hat{z})^T \in (\R^2)^2$
  which rotates with the speed of the true anomaly. The transformation matrix is given by the rotation
  matrix $R(\th) = (
  \begin{smallmatrix}
         \cos\th &-\sin\th \\\sin \th & \cos\th
  \end{smallmatrix})$. The generating function of this
  transformation is given by
  \be  \hat{F}(t, P, \hat{z})
         = -P\cdot R(\th)\hat{z},  \lb{P-02}\ee
  and for $1\le i\le n-1$ the transformation is given by
  \be
  q = -\frac{\pt \hat{F}}{\pt P} = R(\th)\hat{z},  \quad 
  \hat{Z} = -\frac{\pt \hat{F}}{\pt \hat{z}} = R(\th)^TP,  \lb{P-03}. \ee
  Writing $\dot{R}(\th(t))=\frac{d}{dt}R(\th(t))$, and noting that $R(\th)^T=R(\th)^{-1}$ and
  $\dot{R}(\th)=\dot{\th}JR(\th)$ we obtain the function
  \begin{align}
    \hat{F}_t
  \equiv \frac{\pt \hat{F}}{\pt t} = -P\cdot\dot{R}(\th)\hat{z} 
  = -\hat{Z}\cdot R(\th)^T\dot{R}(\th)\hat{z} 
  = -\dot{\th}\left(\hat{Z}\cdot R(\th)^TJR(\th)\hat{z}\right) 
  = \dot{\th}(\hat{z}\cdot J\hat{Z}).
  \end{align}
  
  Because $\th=\th(t)$ depends on $t$, by adding the function
  $\frac{\pt \hat{F}}{\pt t}$ to the Hamiltonian function $H$ in \eqref{HaFun}, as in Line 5 in p.272 of \cite{MS}, we obtain the Hamiltonian function $\hat{H}$ in the
  new coordinates:
  \be
  \hat{H}(t,\hat{Z},\hat{z}) = H_0(P,q) + \hat{F}_t  
  = \frac{1}{2}|\hat{Z}|^2 + (\hat{z}\cdot J\hat{Z})\dot{\th}
               - \sum_{i=1}^{N-1}{m_i\over|r(t)a_i-\hat{z}|},   \lb{P-07}
  \ee
  where the variables of $H_0$ are functions of $\th$, $\hat{Z}$, $\hat{z}$,  given by \eqref{P-03}.

  {\bf Step 2.} {\it Dilating coordinates via the polar radius $r=|z(t)|$.}
  
  We change the coordinates $\hat{\xi}$ to
  $\td{\xi}=(\td{Z},\td{z})$
  which dilate with $r=|z(t)|$ given by \eqref{rTh}. The position coordinates are transformed by $\hat{z} = r\td{z}$.
  It is natural to scale the momenta by $1/r$ to get $\hat{Z}=\td{Z}/r$. 
  But it
  turns out that the new transformation
  \be  \hat{Z}=\frac{1}{r}\td{Z}+\dot{r}\td{z}  \lb{P-09}\ee
  makes the resulting Hamiltonian function simpler. 
  This transformation is generated by the function
  \be  \td{F}(t, \td{Z}, \hat{z})
    = \frac{1}{r}\td{Z}\cdot\hat{z} + \frac{\dot{r}}{2r}|\hat{z}|^2,
                          \lb{P-10}\ee
  and is given by
  \be
  \td{z} = \frac{\pt \td{F}}{\pt \td{Z}} = \frac{1}{r}\hat{z},   
  \hat{Z} = \frac{\pt \td{F}}{\pt \hat{z}} = \frac{1}{r}\td{Z}+\frac{\dot{r}}{r}\hat{z} = \frac{1}{r}\td{Z}+\dot{r}\td{z}, \nn\ee
  with
  \be  \frac{\pt \td{F}}{\pt t}
  = -\frac{\dot{r}}{r^2}\td{Z}\cdot\hat{z}
         + \frac{\ddot{r}r-\dot{r}^2}{2r^2}|\hat{z}|^2
  = -\frac{\dot{r}}{r}\td{Z}\cdot\td{z}
         + \frac{\ddot{r}r-\dot{r}^2}{2}|\td{z}|^2,     \lb{P-11}\ee
  by \eqref{P-09}.
  
  In this case, as in the last two lines on p.272 of \cite{MS}, the
  Hamiltonian function $\hat{H}$ in \eqref{P-07}
  becomes the new Hamiltonian function $\td{H}$ in the new coordinates:
  \bea
  \td{H}(t,\td{Z},\td{z})
  &\equiv& \hat{H}(t, \hat{Z}, \hat{z}) + \td{F}_t \nn\\
  &=& \frac{1}{2r^2}|\td{Z}|^2 + \frac{\dot{r}}{r}\td{Z}\cdot\td{z}
     + \frac{\dot{r}^2}{2}(|\td{z}|^2) + (\td{z}\cdot J\td{Z})\dot{\th}
     - \sum_{i=1}^{N-1}{m_i\over|r(t)a_i-r(t)\tilde{z}|} + \tilde{F}_t   \nn\\
  &=& \frac{1}{2r^2}|\td{Z}|^2+ \frac{r\ddot{r}}{2}|\td{z}|^2
     +(\td{z}\cdot J\td{Z})\dot{\th}
     - \frac{1}{r}\sum_{i=1}^{N-1}{m_i\over|a_i-\tilde{z}|}.   \lb{P-12}\eea

  {\bf Step 3.} {\it Coordinates via the true anomaly $\th$ as the independent variable.}
  
  Here we use the true anomaly $\th\in [0,2\pi]$ as an independent variable instead of $t\in [0,T]$ to simplify
  the study. This is achieved by dividing the Hamiltonian function $\td{H}$ in \eqref{P-12} by $\dot{\th}$. Assuming
  $\dot{\th}(t)>0$ for all $t\in [0,T]$ and  $\td{\xi}\in W^{1,2}(\R/(T\Z), \R^8)$ we consider the action functional
  corresponding to the Hamiltonian system:
  \bea f(\td{\xi})
  &=& \int_0^T(\frac{1}{2}\dot{\td{\xi}}(t)\cdot J\td{\xi}(t) - \td{H}(t,\td{\xi}(t))) dt  \nn\\
  &=& \int_0^{2\pi}\left(\frac{1}{2}\frac{\dot{\td{\xi}}(t(\th))}{\dot{\th}(t)}\cdot J\td{\xi}(t)
            - \frac{\td{H}(t,\td{\xi}(t(\th)))}{\dot{\th}(t)}\right) d\th  \nn\\
  &=& \int_0^{2\pi}\left(\frac{1}{2}\td{\xi}'(\th)\cdot J\td{\xi}(\th) - \td{H}(\th,\td{\xi}(\th))\right) d\th.  \nn\eea
  Here we used $\td{\xi}'(\th)$ to denote the derivative of $\td{\xi}(\th)$ with respect to $\th$. But
  in the following we shall still write $\dot{\td{\xi}}(\th)$ for the derivative with respect to $\th$ instead of
  $\td{\xi}'(\th)$ for notational simplicity.
  
  Note that the elliptic Kepler orbit \eqref{rTh} satisfies
  $ r(t)^2\dot{\th}(t) = \sqrt{\mu p} = \sqrt{\mu a(1-e^2)} = \sg^2$ with $\sg=(\mu p)^{1/4}.$
  Note that $a=\mu^{1/3}(T/2\pi)^{2/3}$ with $T$ being the minimal period of the orbit \eqref{rTh}, we have
  $$  \sg = (\mu a(1-e^2))^{1/4} = \mu^{1/3}(\frac{T}{2\pi})^{1/6}(1-e^2)^{1/4}\;\in\; (0,\mu^{1/3}(\frac{T}{2\pi})^{1/6}] $$
  depending on $e$, when the mass $\mu$ and the period $T$ are fixed. Note that similarly we have $p=\sg^4/\mu$ depends
  on $e$ too. Note that the function $r$ satisfies
  $$  \ddot{r} = \frac{\mu p}{r^3} - \frac{\mu}{r^2} = \mu\left(\frac{p}{r^3} - \frac{1}{r^2}\right).  $$
  
  Therefore we get the Hamiltonian function $\td{H}$ in the new coordinates:
  \bea
  \td{H}(\th,\td{Z},\td{z})
       &\equiv& \frac{1}{\dot{\th}}\td{H}(t,\td{Z},\td{z})   \nn\\
  &=& \frac{1}{2r^2(t)\dot{\th}(t)}|\td{Z}|^2
         + \frac{r(t)\ddot{r}(t)}{2\dot{\th}(t)}|\td{z}|^2 
         + \td{z}\cdot J\td{Z}
         - \frac{1}{r(t)\dot{\th}(t)}\sum_{i=1}^{N-1}{m_i\over|a_i-\tilde{z}|}  \nn\\
  &=& \frac{1}{2\sg^2}|\td{Z}|^2+ \td{z}\cdot J\td{Z}
         + \frac{\mu(p-r(\th))}{2\sg^2}|\td{z}|^2
         - \frac{r(\th)}{\sg^2}\sum_{i=1}^{N-1}{m_i\over|a_i-\tilde{z}|},  \lb{P-13}\eea
  where $r(\th)=p/(1+e\cos\th)$. Note that now the minimal period $T$ of the elliptic solution $\td{z}=\td{z}(\th)$
  becomes $2\pi$ in the new coordinates in terms of true anomaly $\th$ as an independent variable.

  {\bf Step 4.} {\it Coordinates via the dilation of $\sg=(p\mu)^{1/4}$.}
  
  The last transformation is the dilation $(\td{Z},\td{z}) \;=\;
        (\sg\bar{Z},\sg^{-1}\bar{z})$.
  This transformation is symplectic and independent of the true anomaly $\th$. Thus the
  Hamiltonian function $\td{H}$ in \eqref{P-13} becomes a new Hamiltonian function:
  \be
  H(\th,\bar{Z},\bar{z}) \equiv
      \td{H}(\th,\sg\bar{Z},\sg^{-1}\bar{z})  
      = \frac{1}{2}|\bar{Z}|^2 + \bar{z}\cdot J\bar{Z}
        + \frac{p-r}{2p}|\bar{z}|^2
        - \frac{r}{\sg}\sum_{i=1}^{N-1}{m_i\over|\sg a_i-\bar{z}|}.  \lb{P-15}
  \ee
  The proof is complete. 
\end{proof}

Suppose that  $(P(t),q(t))^T$ is the solution of  system \eqref{1.2} with
$
q(t)=r(t)R(\theta(t))a_{N}$, and $P(t)=\dot{q}(t).$
By Proposition  \ref{P2.1}, it is transformed to the new solution $(\bar{Z}(\theta),\bar{z}(\theta))^T$ in the true anomaly $\theta$ as the new
variable for the original Hamiltonian function $H$ of \eqref{HaFun}, which is
given by
\be
\bar{Z}(\theta)
=(0,  \sigma)^T
\qquad
\bar{z}(\theta) = 
(\sigma, 0)^T
\lb{sigma-solution}\ee
Therefore, we can prove the Theorem \ref{linearized.Hamiltonian} directly. 

\begin{proof}[Proof of Theorem \ref{linearized.Hamiltonian}]
       In this proof we omit all the upper bars on the variables of $H$ in \eqref{eqn:red.Ham}. By \eqref{eqn:red.Ham}, we have
       \begin{align}
        H_z=JZ+\frac{p-r}{p}z
        -\frac{r}{\sigma}{\partial\over\partial z}\left(\sum_{i=1}^{N-1}{m_i\over|\sg a_i-\bar{z}|}\right)|_{z=\sg a_{N}},  \quad H_{Z} =Z-Jz, 
       \end{align}
       and
       \begin{align}
        H_{ZZ}=I,\quad H_{Zz}=-J, \quad H_{zz}=J, \quad 
        H_{zz}=\frac{p-r}{p}I-\frac{r}{\sigma}{\partial^2\over\partial z^2}\left(\sum_{i=1}^{N-1}{m_i\over|\sg a_i-\bar{z}|}\right)|_{z=\sg a_{N}},
       \end{align}
       where all the items above are $2\times2$ matrices, and we denote by $H_x$ and $H_{xy}$
       the derivative of $H$ with respect to $x$, and the second derivative of $H$ with respect to
       $x$ and then $y$ respectively for $x$ and $y\in\R$.
       
       Now evaluating the corresponding functions at the special solution
       $({0,\sg},{\sg,0})^T\in\R^{4}$
       of \eqref{sigma-solution} with $z=(\sg,0)^T$, and summing them up, we
       obtain
       \begin{eqnarray}
       \frac{\partial^2}{\partial z^2}\left(\sum_{i=1}^{N-1}{m_i\over|\sg a_i-\bar{z}|}\right)\Bigg|_{\xi_0}
       &=&-\left(-\sum_{i=1}^{N-1}{m_i\over|\sigma a_i-z|^3}\right)I
          +3\sum_{i=1}^{N-1} m_i\frac{(\sg a_i-z)(\sg a_i-z)^T}{|\sigma a_i-z|^5}
       \nonumber\\
       &=&-{1\over\sg^3}\left(\sum_{i=1}^{N-1}{m_i\over|a_i-a_{N}|^3}\right)I
       +{3\over\sg^3}\sum_{i=1}^{N-1} m_i\frac{(a_i-a_{N})(a_i-a_{N})^T}{|a_i-a_{N}|^5}.
       \label{U_zz}
       \end{eqnarray}
       Hence, we have
       \begin{eqnarray}
       H_{zz}&=&I_2-{r\over p}I_2-{r\over p}{\sg^3\over\mu}
                  \left[-{1\over\sg^3}\left(\sum_{i=1}^{N-1}{m_i\over|a_i-a_{N}|^3}\right)I
                  +{3\over\sg^3}\sum_{i=1}^{N-1} m_i\frac{(a_i-a_{N+1})(a_i-a_{N})^T}{|a_i-a_{N}|^5}\right]
                  \nn\\
             &=&I_2-{r\over p}\left[I_2-{1\over\mu}\left(\sum_{i=1}^{N-1}{m_i\over|a_i-a_{N}|^3}\right)I
             +{3\over\mu}\sum_{i=1}^{N-1} m_i\frac{(a_i-a_{N})(a_i-a_{N})^T}{|a_i-a_{N}|^5}\right],
       \end{eqnarray}
       where the last equality of the first formula follows from the definition of $\mu$ in \eqref{eq.of.cc}.
       Thus the proof is complete. 
\end{proof}

\setcounter{equation}{0}
\section{The linear stability restricted $4$-body problem}\label{sec:3}

We now consider the linear stability of special relative equilibrium in the four body problem with one small mass away form the line of the three primaries which form an Euler central configuration. 
A typical example is the ERE orbit of the restricted
$4$-bodies, the Sun, the Earth, the Moon and one space station. 

\begin{proposition}\label{prop.sum.lm3.lm4}
  Suppose that $a_i$ for $1\leq i \leq 3$ forms an Euler collinear configuration and $a_i$ for $1 \leq i \leq 4$ span $\R^2$. The eigenvalues $\lm_3$ and $\lm_4$ of $D$ defined by \eqref{matrix.D} are both positive and satisfy
  \begin{align}
    \lm_3 + \lm_4 = 3.
  \end{align}
\end{proposition}
\begin{proof}
By \eqref{matrix.D} and the direct computations, we have that 
\begin{align}\lb{eqn:sum.l3.l4}
  \lambda_3+\lambda_4 =tr(D)
  =2+{1\over\mu}\sum_{i=1}^3{m_i\over|a_i-a_4|^3} 
  =3+\beta_{2,0},
  \end{align}
where $\beta_{2,0}={1\over\mu}\sum_{i=1}^3{m_i\over|a_i-a_4|^3}-1.$
Moreover, we have
\begin{align}
  D-{3+\beta_{2,0}\over2}I_2
&=-{3\over2\mu}\sum_{i=1}^3{m_i\over|a_i-a_4|^3}I_2
 +{3\over\mu}\sum_{i=1}^3{m_i(a_i-a_4)(a_i-a_4)^T\over|a_i-a_4|^5}
 \\
&={3\over2\mu}\left[\sum_{i=1}^3m_i{2(a_i-a_4)(a_i-a_4)^T-|a_i-a_4|^2\over|a_i-a_4|^5}\right]\\
&=\Psi(\beta_{22,0}),
\end{align}
where
$\beta_{22,0}=\sum_{i=1}^3m_i{(z_{a_i}-z_{a_4})^2\over|a_i-a_4|^5}$
and $\Psi(z) = (\begin{smallmatrix}
  x & y\\-y & x
\end{smallmatrix})$ with $z = x+ \sqrt{-1} y$.
Note that  $\beta_{2,0}$ and $\beta_{22,0}$ are coincide with
those of (2.10) and (A.14) in \cite{LiZ}.
Therefore, following the discussion in \cite{LiZ},we have
\begin{align}\lb{eqn:l3l4}
  \lambda_3,\lambda_4={3+\beta_{22,0}\over2}\pm|\beta_{22,0}|.
\end{align}

Now suppose $a_1$, $a_2$and $a_3$ form an Euler collinear configuration where they are locate on the $x$-axis in $\R^2$. 
We set $a_i=(a_{ix},a_{iy})$ for $i=1,2,3,4$,
and hence $a_{iy}=0$ when $i=1,2,3$.
Then the central configuration equation \eqref{eq.of.cc} of $m_4$ gives
$
\sum_{j=1}^3{m_j(a_j-a_4)\over|a_j-a_4|^3}=-\mu a_4.
$
Especially, for $a_{4y}$, we have that 
\begin{equation}
\sum_{j=1}^3{m_ja_{4y}\over|a_j-a_4|^3}=\mu a_{4y}.
\end{equation}
Since $a_{4y}\ne0$, it follows that 
\begin{equation}\label{second.formula.of.mu}
\mu=\sum_{j=1}^3{m_j\over|a_j-a_4|^3},
\end{equation}
and hence $\beta_{2,0}=0$. By \eqref{eqn:l3l4}, we have
$\lambda_3+\lambda_4=3$
and
\begin{align}
D={3\over\mu}\sum_{j=1}^3{m_j(a_j-a_4)(a_j-a_4)^T\over|a_j-a_4|^5}
={3\over\mu}
\begin{pmatrix}
       \sum_{i=1}^3{m_ix_{4i}^2\over r_{4i}^5} &
	\sum_{i=1}^3{m_ix_{4i}y_{4i}\over r_{4i}^5}\\
	\sum_{i=1}^3{m_ix_{4i}y_{4i}\over r_{4i}^5} &
	\sum_{i=1}^3{m_iy_{4i}^2\over r_{4i}^5} 
\end{pmatrix},\lb{eqn:D}
\end{align}
where $(x_{4i},y_{4i})=a_4-a_i$ and $r_{4i}=|a_4-a_i|$.
Using Cauchy's inequality, we have
\begin{equation}
\left(\sum_{i=1}^3{m_ix_{4i}y_{4i}\over r_{4i}^5}\right)^2
=\left(\sum_{i=1}^3\sqrt{m_ix_{4i}^2\over r_{4i}^5}\sqrt{m_iy_{4i}^2\over r_{4i}^5}\right)^2
\le\left(\sum_{i=1}^3{m_ix_{4i}^2\over r_{4i}^5}\right)
   \left(\sum_{i=1}^3{m_iy_{4i}^2\over r_{4i}^5}\right).
\end{equation}
Then we have $\det(D)\ge0$ and hence $\lambda_3,\lambda_4 \geq 0$. 
\end{proof}

Without loss of the generality, we assume that $\lambda_3 \geq \lambda_4$ in the following.
We define the operator $\cA(\aa, e)$ by
\begin{align}\lb{eqn:ca}
  \cA(\aa, e) &= -\frac{\d^2}{\d t^2}I_2  -I_2 + \frac{1}{1+e\cos t}R(t)D_{\alpha,e}R(t)^T \nn \\
  &= -\frac{\d^2}{\d t^2}I_2  -I_2 + \frac{1}{2(1+e\cos t)}(
  (\lambda_3 + \lambda_4)I_2 + (\lambda_3 - \lambda_4) S(t))\\
  &= -\frac{\d^2}{\d t^2}I_2  -I_2 + \frac{1}{2(1+e\cos t)}(
       3I_2 + \alpha S(t)),  \lb{cA}
\end{align}
where $R(t) = (\begin{smallmatrix}\cos t & -\sin t\\ \sin t & \cos t \end{smallmatrix})$,
$D_{\bb,e}=
(\begin{smallmatrix}
\lambda_3 & 0\\ 0 & \lambda_4
\end{smallmatrix})$,
$S(t) = (\begin{smallmatrix}
\cos 2t & \sin 2t\\ \sin 2t & -\cos 2t
\end{smallmatrix})$ and  $\alpha:= \lambda_3 - \lambda_4 \in [0,3]$.
We denote the Morse indices and nullity of 
$\cA(\aa,e)$ on the domain $\ol{D(\om,2\pi)}$ by $\phi_{\om} (\cA(\aa,e))$ and $\nu_{\om} (\cA(\aa,e))$ respectively. We also denote 
the $\om$-Maslov index of $\xi_{\aa,e}$ by $i_{\om}(\xi_{\aa,e})$ and the nullity by $\nu_{\om}(\xi_{\aa,e})$.
By Lemma \ref{lem:2.1} and the transformation introduced by Section 2.4 of \cite{HLS}, 
we have for any $\aa \in [0,3]$, $e\in[0,1)$ and $\om\in \U$,
\begin{align}
  \phi_{\om} (\cA(\aa,e)) = i_{\om}(\xi_{\aa,e}), \quad \nu_{\om} (\cA(\aa,e)) = \nu_{\om}(\xi_{\aa,e}). \lb{eqn:ind.equ}
\end{align}

According to (2.17) and (2.18) of \cite{HLS}, the essential part $\gamma=\gamma_{\beta, e}(t)$ of the fundamental solution of the Lagrangian orbit satisfies
\begin{align}\lb{eqn:lagga}
\begin{cases}
&\dot{\gamma}(t)=J B(t) \gamma(t), \\
&\gamma(0)=I_{4},
\end{cases}
\end{align}
with
$B(t)=\left(\begin{smallmatrix}
I & -J \\
J & I_2 - \frac{K_{\beta}}{1+e\cos t}
\end{smallmatrix}\right),
$
where $e$ is the eccentricity and $K_{\beta}=\diag\{\frac{3+\sqrt{9-\beta} }{2}, \frac{3-\sqrt{9-\beta}}{2}\}$.
 For $(\beta, e) \in[0,9) \times[0,1)$, the second order differential operator corresponding to (\ref{eqn:lagga}) is given by
 \begin{align}
 A(\beta, e) =-\frac{\mathrm{d}^{2}}{\mathrm{~d} t^{2}} I_{2}-I_{2}+\frac{1}{2(1+e \cos t)}\left(3 I_{2}+\sqrt{9-\beta} S(t)\right), \lb{eqn:lagA}
 \end{align}
 where $S(t)=\left(\begin{smallmatrix}\cos 2 t & \sin 2 t \\ \sin 2 t & -\cos 2 t\end{smallmatrix}\right)$, defined on the domain $\ol{D(\om,2\pi)}$ in \eqref{A.4}. Then it is self-adjoint and depends on the parameters $\beta$ and $e$. 
Let $\beta := 9 - \alpha^2$ in \eqref{eqn:lagA}, we have that 
\begin{align}\lb{eqn:aca}
       A(9-\alpha^2,e ) = \cA(\alpha ,e).
\end{align}
It follows that the Maslov index of $i_{\omega}(\xi_{\alpha ,e}) = i_{\omega}(\gamma_{9-\alpha^2,e})$ by letting $\beta = 9 - \alpha^2$ in \eqref{eqn:lagga}.
By \cite{HLS}, the Maslov index of the $\xi_{\alpha ,e}$ is given as follows.
\begin{proposition}\lb{prop:index}
\begin{enumerate}[label=(\roman*)]
       \item For all $(\alpha,e) \in [0,3] \times [0,1)$, the Maslov index satisfies $i_1(\xi_{\alpha ,e}) = 0;$
       \item for $e\in [0,1)$, 
       \begin{align}
              i_{\omega}\left(\xi_{3, e}\right) =
              \begin{cases}
                     2, \mbox{ if } \om \in\U \bs \{0\},\\
                     0, \mbox{ if } \om = 1;
              \end{cases}
              \quad 
              v_{\omega}\left(\xi_{3, e}\right)=
              \begin{cases}3, & \text { if } \omega=1, \\ 
              0, & \text { if } \omega \in \mathbf{U} \backslash\{1\};
       \end{cases}
       \end{align}
       \item  for $\omega \in \mathbf{U}$ and $e\in [0,1)$, we have 
       $i_{\omega}\left(\xi_{0, e}\right) =0,$ and $
       v_{\omega}\left(\xi_{0, e}\right)=0.$
\end{enumerate}
\end{proposition}
\begin{proof}
       By Theorem 1 of \cite{HLS} and \eqref{eqn:aca}, we have (i) of this proposition holds. By (C) of Section 3 of \cite{HLS}, we have (ii) of this proposition holds. By Section 3.2 of \cite{HLS}, we have (iii) of this proposition holds.
\end{proof}

To study the monotonic of the operator $\cA(\aa, e)$, we define the operator $\bar{\cA}$ as the following.
\begin{align}
       \bar{\cA}(\aa, e)=\frac{\cA(0, e)}{\aa}+\frac{S(t)}{2(1+e \cos t)}. \lb{eqn:barA}
\end{align}
It follows that $\cA(\aa ,e) =\aa \bar{\cA}(\aa, e).$
Then we have $\phi_{\omega}(\cA(\aa, e)) =\phi_{\omega}(\bar{\cA}(\aa, e))$ and $\nu_{\omega}(\cA(\aa, e)) =\nu_{\omega}(\bar{\cA}(\aa, e))$.
Following from (iii) of Proposition \ref{prop:index} and Lemma \ref{lem:2.1}, $\cA(0, e)$ is positive definite for any $\omega$ boundary condition. Using similar arguments of Lemma 4.4 and Corollary 4.5 of \cite{HLS},we get the following lemma holds. We omit the proof here.
\begin{proposition}[cf. Lemma 4.4 and Corollary 4.5 of \cite{HLS}]\lb{prop:index}
\begin{enumerate}[label = (\roman*)]
\item For each fixed $e \in[0,1)$, the operator $\bar{\cA}(\aa, e)$ is non-increasing with respect to $\aa \in(0,3]$ for any fixed $\omega \in \mathbf{U}$. Especially
$
\left.\frac{\partial}{\partial \aa} \bar{\cA}(\aa, e)\right|_{\aa=\aa_{0}}= - \frac{1}{\aa^2} \cA(0, e),
$
for $\aa \in(0,3]$ is a negative definite operator.
\item  For every eigenvalue $\lambda_{\aa_{0}}=0$ of $\bar{A}\left(\aa_{0}, e_{0}\right)$ with $\omega \in \mathbf{U}$ for some $\left(\aa_{0}, e_{0}\right) \in$ $(0,3] \times[0,1)$, there holds
\begin{align}\lb{eqn:mono.barA}
\left.\frac{\mathrm{d}}{\mathrm{d} \aa} \lambda_{\aa}\right|_{\aa=\aa_{0}}>0.
\end{align}
\item For every fixed $e \in[0,1)$ and $\omega \in \mathbf{U}$, the index function  $i_{\omega}\left(\xi_{\aa, e}\right)$, is non-decreasing in $\alpha \in[0,3]$. Especially, if $\omega \in \mathbf{U} \backslash\{1\}$, $i_{\omega}\left(\xi_{\aa, e}\right)$ increases from $0$ to $2$.
\end{enumerate}
      
\end{proposition}

By Proposition \ref{prop:index}, the $-1$-index $i_{-1}\left(\xi_{\alpha, e}\right)$ changes monotonically increases from $0$ to $2$ as $\aa$ increases from $0$ to $3$. Then there exists three distinct curves $\Gamma_{k}$, $\Gamma_{s}$,  and $\Gamma_{m}$  locating from left to right in the parameter $(\alpha, e)$ rectangle $[0,3] \times[0,1)$. The curves $\Gamma_{s}$ and $\Gamma_{m}$ are the $-1$-degenerate curves. Since the $i_{\om}(\xi_{\aa,e})$ is non-decreasing, the hyperbolic region is connected. We then can use the curve $\Gamma_{k}$ to denote the boundary of the hyperbolic region. 

More specifically, for every $e \in[0,1)$, the $-1$ index $i_{-1}\left(\xi_{\alpha, e}\right)$ is non-increasing, and strictly decreasing only on two values of $\alpha=\alpha_{1}(e)$ and $\alpha=\alpha_{2}(e) \in(0,3)$. Define 
\begin{align}\lb{eqn:aa.sm}
       \alpha_{s}(e)=\min \left\{\alpha_{1}(e), \alpha_{2}(e)\right\} \text { and } \alpha_{m}(e)=\max \left\{\alpha_{1}(e), \alpha_{2}(e)\right\}, \quad \text { for } e \in[0,1), 
\end{align}
and
\begin{align}
\Gamma_{s}=\left\{\left(\alpha_{s}(e), e\right) \mid e \in[0,1)\right\}, \text { and } \Gamma_{m}=\left\{\left(\alpha_{m}(e), e\right), \mid e \in[0,1)\right\}. \lb{eqn:gas}
\end{align}
For every $e \in[0,1)$, we define
\begin{align}
       \alpha_{k}(e)=\inf \left\{\alpha^{\prime} \in[0,3] \mid \sigma\left(\xi_{\alpha, e}(2 \pi)\right) \cap \mathbf{U} \neq \emptyset, \quad \forall \alpha \in\left[0, \alpha^{\prime}\right]\right\}, \lb{eqn:def.aak}
\end{align}
and
\begin{align}
\Gamma_{k}=\left\{\left(\alpha_{k}(e), e\right) \in[0,3] \times[0,1) \mid e \in[0,1)\right\}.
\end{align}
Therefore, 
$\Gamma_{s}, \Gamma_{m}$ and $\Gamma_{k}$ form three curves.

\begin{lemma}\lb{lem:mono.ind.nul}
       (i) For given $e_0$, if
       $(\aa_1,e_0)$ and $(\aa_1, e_0)$ satisfy that  $0<\aa_2 \leq \aa_1 < 3$
       and $\xi_{\aa_1,e_0}(2\pi)$ is hyperbolic, then $\xi_{\aa_2,e_0}(2\pi)$ is also hyperbolic.
       Consequently, the hyperbolic region of $\xi_{\aa, e}$ is connected.
       
       (ii) For $(\aa,e) \in \Gamma_h$, every matrix $\xi_{\aa, e}(2\pi)$ is hyperbolic.
       Thus $\Gamma_k$ is the boundary surface of this hyperbolic region.
       
       (iii) For any $e\in [0,1)$, 
       the total multiplicity of $\om$ degeneracy  of $\xi_{\aa, e}(2 \pi)$
       for $\aa \in[0,3]$ satisfies that $\sum_{ 0\leq \aa\leq 3} \nu_{\om}(\xi_{\aa, e}(2\pi))=2,$ for all $ \omega \in \mathbf{U} \backslash\{1\}.$
       \end{lemma}
       
       \begin{proof}
       (i) By Lemma \ref{lem:2.1}, for any fixed $\om$ and $e$, 
       if $\xi_{\aa_1, e}(2\pi)$ is hyperbolic, then $\cA(\aa_1,e)$ is positive definite on $\overline{D}(\om ,2\pi)$ for any given $\om \in \U$. By \eqref{eqn:barA}, $\bar{\cA}(\aa_1,e)$ is positive definite. By \eqref{eqn:mono.barA}, we have 
       $\bar{\cA}(\aa_1,e) < \bar{\cA}(\aa_2,e)$.  It follows that $\bar{\cA}(\aa_2,e)$ is positive definite and non-degenerate for any  $\om \in \U$. Therefore $\xi_{\aa_2, e}(2\pi)$ must be hyperbolic for all $\aa \in [0,\aa_2)$.
       
       Note that when $\aa = 0$ the matrix $\xi_{\aa, e}$ is hyperbolic by (iii) of Proposition \ref{prop:index}.
       Therefore, the hyperbolic region of $\xi_{\aa, e}$ is connected.
       
       (ii) By the definition of $\aa_k(e)$, there exists a sequence$\{\aa_i\}_{i\in \N}$ satisfying $\aa_i<\aa_k(e)$, $\aa_i\to \aa_k(e)$, and $\xi_{\aa_i,e}(2\pi)$ is hyperbolic.
       Therefore $\xi_{\aa,e}(2\pi)$ is hyperbolic
       for every $\aa\in [0,\aa_k(e))$ by (i).
       Then \eqref{eqn:def.aak} holds and $\aa_k(e)$ is the envelope surface of this hyperbolic region.
       
       (iii) Note that the operator $\bar{\cA}(\aa,e)$ on $\overline{D}(\om, 2\pi)$ for $0$ and $3$ are both non-degenerate if $\om\in \U\bs\{1\}$.
       By (ii) and (iii) of Proposition \ref{prop:index} and \eqref{eqn:mono.barA}, there exist at most two $\aa_1$ and $\aa_2$ such that
       at each of which the $\om$-index decreases
       by 1 if $\aa_1 \neq \aa_2$, or the $\om$-index decreases by 2 if $\aa_1 =\aa_2$. Suppose that the two values are given by $\aa_1 = \aa_1(e)$
       and $\aa_2 = \aa_2(e)$ such that for $\ep > 0$ small enough, we have that
       $\phi_{\om}(\bar{A}(0 ,e)) = \phi_{\om}(\bar{A}(\aa_1 - \ep ,e))$, $\phi_{\om}(\bar{\cA}(\aa_1, e)) = \phi_{\om}(\bar{\cA}(\aa_2-\ep,e ))$ and $\phi_{\om}(\bar{\cA}(\aa_2,e )) =\phi_{\om}(\bar{\cA}(3,e ))$.
       Then we have that
       \begin{align}
         2 &=\phi_{\om}(\bar{\cA}(3,e ))-\phi_{\om}(\bar{\cA}(0,e )) \nn\\ 
         &=\phi_{\om}(\bar{\cA}(\aa_2,e ))-\phi_{\om}(\bar{\cA}(\aa_2 - \ep,e )) +\phi_{\om}(\bar{\cA}(\aa_1,e ))-\phi_{\om}(\bar{\cA}(\aa_1-\ep,e )) \nn\\
         &=\dim \ker(\bar{\cA}(\aa_2,e ))+\dim \ker(\bar{\cA}(\aa_1,e )) \nn \\
         &=\nu_{\om}(\xi_{\aa_2, e}(2\pi))+\nu_{\om}(\xi_{\aa_1, e}(2\pi)) \nn\\
         &=\sum_{\aa \in[0,3]} \nu_{\om}(\xi_{\aa, e}(2\pi)).
       \end{align}
       Then we have that (iii) holds.
       \end{proof}

       Following similar arguments as Corollary 9.2 of \cite{HLS}, we have the following corollary holds.
       \begin{corollary}[cf. Corollary 9.2 of \cite{HLS}]\lb{coro:null.sep}
         For every $e\in [0,1)$,  we have $\sum_{\aa \in[0,\aa_m(e))} \nu_{-1}(\xi_{\aa,e}(2\pi)) =0,$ and $\sum_{\aa \in[\aa_m(e), 3)} \nu_{-1}(\xi_{\aa,e}(2\pi))=2.$
       \end{corollary}

\begin{proposition}
    The function $\aa_k(e)$ is continuous in $e$.
\end{proposition}

\begin{proof}
    We prove this proposition by contradiction. Suppose that $\aa_k(e)$ is not continuous in $e$. There must exist some $\hat{e}$ and a sequence $\{e_i\}_{i=1}^{\infty} \subset [0,1)\bs\{ \hat{e}\}$ and $\aa_0 \in [0, 3]$ such that
    $\aa_k(e_i) \to \aa_0 \neq \aa_k(\hat{e})$, if $e_i \to \hat{e}$ as $i \to \infty$.
    We discuss the two cases of the continuity according to the sign of $\aa_0 -\aa_k(\hat{e})$. By the continuity of the eigenvalues of the matrix $\xi_{\aa,e}(2\pi)$ and by \eqref{eqn:def.aak}, $\sg(\xi_{{\aa}_0,\hat{e}}(2\pi))\cap \U \neq \emptyset$ holds.

    By the definition of $\aa_k(\hat{e})$ and (i) of Lemma \ref{lem:mono.ind.nul}, we must have $\aa_k(\hat{e})<\aa_0$.

    Now we suppose $\aa_k(\hat{e})<\aa_0$. By the continuity of $\aa_s(e)$ and the definition of $\aa_0$,
    \begin{align}
      \aa_k(\hat{e}) < \aa_0 \leq \aa_s(\hat{e}).
    \end{align}
    By the definition of $\aa_k(\hat{e})$, let $\om_0\in \sg(\xi_{\aa_k(\hat{e}),\hat{e}}(2\pi))\cap \U$.
    Let $L=\{(\hat{\aa},\hat{e}) | \aa \in [0,\aa_{k}(\hat{e}))\}$, $V=\{(0, e)| e \in[0,1)\}$, and $L_i =\{(\aa ,e_i)|\aa\in [0,\aa_k(e_i)]\}$.
    \begin{align}
      i_{\omega_{0}}\left(\xi_{\aa,e}\right)=
      \nu_{\omega_{0}}\left(\xi_{\aa, e}\right)=0,
      \quad \forall(\aa,e) \in L \cup V \cup \bigcup_{i \geq 1} L_{i}.
    \end{align}
    In particular, we have
    $i_{\omega_{0}}\left(\xi_{\aa_{k}(\hat{e}), \hat{e}}\right)=0$ and $\nu_{\omega_{0}}\left(\xi_{\aa_{k}(\hat{e}), \hat{e}}\right) \geq 1$.
    Therefore, by the definition of $\om_0$, there exists $\hat \aa \in (\aa_k(\hat{e}),\aa_0)$ sufficiently close to $\aa_k (\hat{e})$ such that
    \begin{align}
      i_{\omega_{0}}\left(\xi_{\hat{\aa}, \hat{e}}\right)
      =i_{\omega_{0}}\left(\xi_{\aa_{k}(\hat{e}), \hat{e}}\right)+\nu_{\omega_{0}}\left(\xi_{\aa_{k}(\hat{e}), \hat{e}}(2 \pi)\right) \lb{eqn:accol}
      \geq 1.
    \end{align}
    Note that \eqref{eqn:accol} holds for all $\aa \in (\aa_k(e), \aa_0]$. Also $(\aa, \hat{e})$ is an accumulation point of $\cup_{i\geq 1} L_i$. This yields there exists $(\aa_i,e_i) \in L_i$ such that $\xi_{\aa_i, e_i}$ is $\om_0$-degenerate.
    Moreover $\aa_i\to \aa_0$ and $\xi_{\aa_i,e_i}(2\pi) \to \xi_{\aa_0,\hat e}(2\pi)$ as $i\to \infty$.
    Then we have following contradiction for $i \geq 1$ large enough
    \begin{align}
      1 \leq  i_{\omega_{0}}(\xi_{\aa_0, \hat{e}}) \leq
       i_{\omega_{0}}(\xi_{\aa_i(e), e_i}) = 0.
    \end{align}
    Then we have the continuity of $\aa_k(e)$ in $\aa$ and $e$.
\end{proof}

\begin{proposition}    
The region $\Gamma_{s}, \Gamma_{m}$ and $\Gamma_{k}$ form three curves which possess the following properties.
\begin{enumerate}[label = (\roman*)]
       \item We have
       \begin{align}
       i_{-1}\left(\xi_{\alpha, e}\right)=
       \begin{dcases}
       0, \text { if } 0 \leq \alpha  \leq \alpha_{s}(e), \\
       1, \text { if } \alpha_{s}(e) < \alpha \leq \alpha_{m}(e), \\
       2, \text { if } \alpha_{m}(e) < \alpha \leq 3,
       \end{dcases}
       \end{align}
       and both $\Gamma_{s}$ and $\Gamma_{m}$ are precisely the $-1$ degeneracy curves of the matrix $\xi_{\alpha, e}(2 \pi)$ in the $(\alpha, e)$ rectangle $[0,3] \times[0,1)$.
       \item There holds $\alpha_{k}(e) \leq \alpha_{s}(e) \leq \alpha_{m}(e)<3$ for all $e \in[0,1)$.
       \item Every matrix $\xi_{\alpha, e}(2 \pi)$ is hyperbolic when $\alpha \in\left(0,\alpha_{k}(e)\right]$ and $e \in[0,1)$, and there holds
       \begin{align*}
       \alpha_{k}(e)=\sup \left\{\alpha \in[0,3] \mid \sigma\left(\xi_{\alpha, e}(2 \pi)\right) \cap \mathbf{U}=\emptyset\right\}, \quad \forall e \in[0,1).
       \end{align*}
       Consequently $\Gamma_{k}$ is the boundary curve of the hyperbolic region of $\xi_{\alpha, e}(2 \pi)$ in the $(\alpha, e)$ rectangle $[0,3] \times[0,1)$.
\end{enumerate}
\end{proposition}

\begin{proof}
	For $0\leq e <1 $, by the definitions of $\aa_s(\aa,e)$ and $\aa_m(e)$ satisfying  for $\aa > 0$ and $e \in[0,1)$ in \eqref{eqn:gas},
	we have that $-1$-index stays the same and only changes when $(\aa,e) \in \Gamma_{m}\cup \Gamma_{s}$.

       Note that if $\aa  = 0$, $\sigma\left(\xi_{\aa, e}(2 \pi)\right) \cap \mathbf{U} = \emptyset$. Therefore for any $e\in[0, 1)$, $\aa_{k}(e)$ is well defined.
By the definition of \eqref{eqn:gas} and \eqref{eqn:def.aak}, we have that for given $e_0\in[0,1)$,  if $\aa_s(e_0)$ exists,  we must have $ \aa_k(e_0) \leq \aa_s(e_0) \leq \aa_m(e_0)$.
       By Lemma \ref{lem:mono.ind.nul}, (ii) and (iii) of this propositon can be obtianed directly.
	Then we have this proposition holds.
\end{proof}

\begin{theorem}\lb{thm:RE.norm.form}
\begin{enumerate}[label = (\roman*)]
       \item 
       We have $\xi_{\alpha, e}(2 \pi) \approx R\left(\theta_{1}\right) \diamond R\left(\theta_{2}\right)$ for some $\theta_{1}$ and $\theta_{2} \in(\pi, 2 \pi)$, and thus it is strongly linearly stable on the segment $\alpha_{m}(e)<\alpha<3$; 
       \item We have $\left.\xi_{\alpha, e}(2 \pi) \approx D(\lambda) \diamond R(\theta)\right)$ for some $0>\lambda \neq-1$ and $\theta \in(\pi, 2 \pi)$ and it is elliptic-hyperbolic, and thus linearly unstable on the segment $\alpha_{s}(e)<\alpha<\alpha_{m}(e)$.
       \item  We have $\xi_{\alpha, e}(2 \pi) \approx R\left(\theta_{1}\right) \diamond R\left(\theta_{2}\right)$ for some $\theta_{1} \in(0, \pi)$ and $\theta_{2} \in(\pi, 2 \pi)$ with $2 \pi-\theta_{2}<\theta_{1}$, and thus it is strongly linearly stable on the segment $\alpha_{k}(e)<\alpha<\aa_s(e)$.
\end{enumerate}
\end{theorem}

\begin{proof}
    (i) By Lemma \ref{lem:Bott}, the index and nullity of $2$nd-iteration of the symplectic path $\xi_{\aa,e}(t)$ satisfies
    \begin{align}
      i_{1}(\xi_{\aa,e}^{2}) &= i_{-1}\left(\xi_{\aa,e}\right)+i_{1}(\xi_{\aa,e}),\nn\\
      \nu_{1}\left(\xi_{\aa,e}^{2}\right)&=\nu_{1}\left(\xi_{\aa,e}\right)+ \nu_{-1}\left(\xi_{\aa,e}\right)=\nu_{-1}\left(\xi_{\aa,e}\right),
    \end{align}
    where $\nu_{1}(\xi_{\aa,e}) = 0$ if $(\aa,e)\in [0,3] \times[0,1) $. Therefore, the 2-iteration of the index is given by
    \begin{align}
      i_1(\xi^2_{\aa,e})=\left\{
      \begin{array}{ll}
          {2,} & {\text { if } (\aa,e) \in \Gamma_m}; \\
          {1,} & {\text { if } (\aa,e) \in \Gamma_s}; \\
          {0,} & {\text { if } (\aa,e) \in \Gamma_k }.
      \end{array}\right.
    \end{align}
    Follow the discussion in the proof of Theorem 1.2 of \cite{HS}, if $(\aa, e)$ satisfies $\aa \neq \aa_{m}(e)$ and $\aa \neq \aa_{s}(e)$, the matrix  $\xi_{\aa,e}(4\pi) =
    \xi_{\aa,e}^{2}(2\pi)$ is non-degenerate with respect with eigenvalue $1$.

    For (i), suppose  $\omega_{i}=e^{\sqrt{-1 \theta_{i}}} \in \sigma\left(\xi_{\aa,e}(2\pi) \cap \mathbf{U}\right)$ with $\th_i \in(0,\pi)$. Note that $i_{-1}(\xi_{\aa,e}) = 2$, and by \eqref{eqn:split},
    \begin{align}
      i_{-1}(\xi_{\aa,e}) &= i_{1}(\xi_{\aa,e})
      +\sum_{\omega_{i}}(S_{\xi_{\aa,e}(2\pi)}^{+} (\omega_{i} )-S_{\xi_{\aa,e}(2\pi)}^{-} (\omega_{i}))+ S_{\xi_{\aa,e}(2\pi)}^{+}(-1) \nn\\
      &=\sum_{\omega_{i}} (S_{\xi_{\aa,e}(2\pi)}^{+} (\omega_{i} )-S_{\xi_{\aa,e}(2\pi)}^{-} (\omega_{i})). \lb{eqn:-1.split}
    \end{align}
    It yields that
    \begin{align}
      2=\sum_{\omega_{i}}\left(S_{\xi_{\aa,e}(2\pi)}^{+}(\omega_{i})-S_{\xi_{\aa,e}(2\pi)}^{-}(\omega_{i})\right) \leq \sum_{\omega_{i}}\left(S_{\xi_{\aa,e}(2\pi)}^{+}(\omega_{i})+S_{\xi_{\aa,e}(2\pi)}^{-}(\omega_{i})\right) \leq 2.
    \end{align}
     Then we have that $S_{\xi_{\aa,e}}^{-}(\omega_{i}) =0$, by the list of splitting number in Section \ref{sec:Prelim}.
     Therefore, there exist the two $\om_1$ and $\om_2$ such that $S_{\xi_{\aa,e}}^{+}(\omega_{i}) =1$.
     Then we have (i) of Theorem \ref{thm:RE.norm.form} holds.

     (ii) Note that $i_1(\xi^2_{\aa,e}) = 1$ implies that $i_{-1}(\xi_{\aa,e}) = 1$.
     Therefore, still by \eqref{eqn:-1.split}, there exists exact one eigenvalue $\omega=e^{\sqrt{-1 \theta_{i}}} \in \sigma\left(\xi_{\aa,e}(2\pi) \cap \mathbf{U}\right)$ for $\th \in (0,\pi)$ with the splitting number $(1,0)$.
     By the splitting number the list of splitting number in Section \ref{sec:Prelim}, we must have $\xi_{\aa,e}(2\pi) \approx D(\lm) \diamond R(\th)$.
     Also note that $i_1(\xi_{\aa,e}) = 0$ implies $\lm < 0$.
     Then we have (ii) of Theorem \ref{thm:RE.norm.form} holds.

     (iii)
 	  Suppose that  $(\aa_{0},e_{0})\in \Gamma_k$.
 	  By (i) and (ii) of Lemma \ref{lem:mono.ind.nul}, the matrix $M\equiv\xi_{\aa_{0},e_{0}}(2\pi)$ is not hyperbolic with at least one pair of on the unit circle $\U$.
 	  Furthermore, by the definition of $\aa_k(e_{0})$ and $\aa_s(e_0)$, $\pm 1 \notin \sg(\xi_{\aa_{0},e_{0}}(2\pi))$.
 	  Suppose that $ \sigma(\xi_{\aa_{0},e_{0}}(2\pi))= \{\lambda_{1} , \lambda_{1}^{-1}, \lambda_{2}, \lambda_{2}^{-1}\}$,
 	with $\lambda_{1}\in \mathbf{U} \backslash \mathbf{R}$
 	and $\lambda_{2}\in(\mathbf{U} \cup \mathbf{R}) \backslash\{ \pm 1,0\}$. We must have $\lambda_{2}\in \mathbf{U} \bs \mathbf{R}$ and
 	$\xi_{\aa_0, e_{0}}(2\pi) \approx R(\th_1) \diamond R(\th_2)$ for $\th_1$, $\th_2 \in (0,\pi) \cup (\pi ,2\pi)$.

 	If not, $\lambda_{2}\in \R\bs\{\pm 1, 0\}$. The normal form is given by  $\xi_{\aa_0,e_0}(2\pi) \approx D(\lm) \diamond R(\th)$ for some $\th \in (0,\pi) \cup (\pi,2\pi)$. Then by \eqref{eqn:split}, we obtain following contradiction.
  \begin{align}
    0 &=i_{-1}\left(\xi_{\aa_0, e_{0}}\right) \nn\\
    &=i_{1}\left(\xi_{\aa_0, e_{0}}\right)+S_{M}^{+}(1)-S_{M}^{-}\left(e^{ \pm \sqrt{-1} \theta}\right)+S_{M}^{+}\left(e^{ \pm \sqrt{-1} \theta}\right)-S_{M}^{-}(-1) \nn\\
   	&=0+0-S_{R(\theta)}^{-}\left(e^{ \pm \sqrt{-1} \theta}\right)+S_{R(\theta)}^{+}\left(e^{ \pm \sqrt{-1} \theta}\right)-0 \nn\\
   	&=\pm 1. \lb{eqn:lm2.split}
  \end{align}
 	Again as \eqref{eqn:lm2.split}, we have that
  \begin{align}
    0&= i_{-1}\left(\xi_{\aa_0, e_{0}}\right) \nn\\
   	&= i_{1}\left(\xi_{\aa_0, e_{0}} \right)+S_{M}^{+}(1)-S_{R\left(\theta_{1}\right)}^{-}\left(e^{ \pm \sqrt{-1} \theta_{1}}\right)+S_{R\left(\theta_{1}\right)}^{+}\left(e^{ \pm \sqrt{-1} \theta_{1}}\right) \\
    &\quad -S_{R\left(\theta_{2}\right)}^{-}\left(e^{ \pm \sqrt{-1} \theta_{2}}\right)+S_{R\left(\theta_{1}\right)}^{+}\left(e^{ \pm \sqrt{-1} \theta_{2}}\right)-S_{M}^{-}(-1) \nn\\
   	&=-S_{R\left(\theta_{1}\right)}^{-}\left(e^{ \pm \sqrt{-1} \theta_{1}}\right)+S_{R\left(\theta_{1}\right)}^{+}\left(e^{ \pm \sqrt{-1} \theta_{1}}\right)-S_{R\left(\theta_{2}\right)}^{-}\left(e^{ \pm \sqrt{-1} \theta_{2}}\right)+S_{R\left(\theta_{2}\right)}^{+}\left(e^{ \pm \sqrt{-1} \theta_{2} }\right).
  \end{align}
 	Note that if $\th_1$ and $\th_2$ locate at the same interval $(0,\pi)$ or $(\pi, 2\pi)$, the right hand side will be $\pm 2$. Thus, we must have that  $\th_1 \in (0,\pi)$ and $\th_2\in (\pi, 2\pi)$.

 	If $\th_1 = 2\pi -\th_2$, following equation holds.
  \begin{align}
    \sum_{\aa_k < \aa \leq  3} \nu_{\omega}(\xi_{\aa_0, e_{0}})
   	\geq \sum_{ \aa_{s}\leq \bb \leq 3} 	\nu_{\omega}(\xi_{\aa_0, e_{0}})
   	+\nu_{\omega}(\xi_{\aa_0, e_{0}}) \geq 1+2,
  \end{align}
 	where $\om = \exp(\sqrt{-1}\th_1)$.
 	This contradicts to (iii) of Lemma \ref{lem:mono.ind.nul}.

 	If $2\pi -\th_2 > \th_1$, for $\om =\exp(\sqrt{-1}\th_1)$, we have that
  \begin{align}
    0 \leq i_{\omega}(\xi_{\aa_0, e_{0}})
   	=i_{1}(\xi_{\aa_0, e_{0}})+S_{M}^{+}(1)-S_{R(\theta_{1})}^{-}(e^{\sqrt{-1} \theta_{1}})
   	=-S_{R(\theta_{1})}^{-}(e^{\sqrt{-1} \theta_{1}})=-1,
  \end{align}
 	where $M = \xi_{\aa_0, e_{0}}(2\pi)$.
 	This contradiction yields that $2\pi -\th_2 <\th_1$.
 	Then we have (ii) of Theorem \ref{thm:RE.norm.form} holds.
\end{proof}

\begin{lemma}\lb{lem:boundary}
    For some $(\aa_0,e) \in [0,3] \times[0,1) $, if $\xi_{\aa_0,e}(2\pi) \approx M_2(-1,c)$ which is given by \eqref{eqn:m2} for $c_1$, $c_2\in \R$, or it possesses the
    basic normal form $N_1(-1,a) \diamond N_1(-1,b)$ for $a$, $b \in \R$, then $\xi_{\aa,e}(2\pi)$ is hyperbolic for all $\aa \in [0,\aa_0)$.
\end{lemma}
\begin{proof}
    Note that the basic normal form of the matrix $M_2(-1,c) $ is either $N_1(-1,a) \diamond N_1(-1, b)$ or $N_1(-1,a)\diamond D(\lm)$ for some $a$, $b\in \R$ and $0 > \lm \neq -1$. Thus for any $\om \in \U\bs\{1\}$, by \eqref{eqn:split}, we obtain
    \begin{align}
      0 \leq i_{\omega}(\xi_{\aa_{0}, e})=i_{1}(\xi_{\aa_{0}, e})+S_{M}^{+}(1)-S_{M}^{-}(\omega)=-S_{M}^{-}(\omega) \leq 0,
    \end{align}
    where $M = \xi_{\aa_{0}, e}(2\pi)$. Then we have that $i_{\omega}(\xi_{\aa_{0}, e}) = 0$ for all $\om \in \U$. Note that $\phi_{\om}(\bar \cA(\aa, e)) = i_{\omega}(\xi_{\aa_{0}, e})$ and $\nu_{\om}(\bar\cA(\aa, e)) = \nu_{\omega}(\xi_{\aa_{0}, e})$.
    Now from $\phi_{\om}(\cA(\aa, e)) = 0$ and the monotonic of eigenvalues of $\bar \cA(\aa, e)$ in Lemma \ref{lem:mono.ind.nul}, we have that $\bar\cA(\aa,e) > 0$ for all $\aa \in [0,\aa_0)$ on $D(\om, 2\pi)$. Therefore, $\nu_{\omega}(\xi_{\aa,e}) = \nu_{\om}(\bar \cA(\aa,e)) = 0$ holds for all $\aa \in [0,\aa_0)$. Then this lemma holds.
\end{proof}

\begin{theorem}\lb{thm:RE.lim}
	When $(\aa,e) \in \Gamma_s$, $\Gamma_m$ or $\Gamma_k$, the normal form and the linear stability of $\xi_{\aa,e}(2\pi)$ satisfy followings.
	\begin{enumerate}[label=(\roman*)]
		\item If $\aa_s(e) < \aa_m(e)$, we have $\xi_{\aa_{m}(e),e}(2\pi) \approx N_1(-1,1) \diamond R(\th)$ for some $\th\in (\pi,2\pi)$. Thus it is spectrally stable and linearly unstable.

		\item\lb{thm:RE.lim:k<s=m} If $\aa_k(e)<\aa_s(e) =\aa_m(e)$, we have $\xi_{\aa_{s}(e),e}(2\pi) \approx -I_2 \diamond R(\th)$ for some $\th\in (\pi,2\pi)$. Thus it is linearly stable, but not strongly linearly stable.

		\item If $\aa_k(e)<\aa_s(e) < \aa_m(e)$, we have $\xi_{\aa_{s}(e),e}(2\pi) \approx N_1(-1,-1)  \diamond R(\th)$ for some $\th\in (\pi,2\pi)$. Thus it is spectrally stable and linearly unstable.

		\item\lb{thm:RE.lim:k<s<=m} If $\aa_k(e)<\aa_s(e) \leq \aa_m(e)$, we have $\xi_{\aa_{k}(e),e}(2\pi) \approx N_2(e^{\sqrt{-1}\th},b)$ for some $\th\in (0,\pi)$ and $b= \begin{pmatrix}
		b_1 & b_2 \\ b_3 & b_4 \end{pmatrix}$
		satisfying $(b_2-b_3)\sin\th> 0$,
		that is, $N_2(e^{\sqrt{-1}\th},b)$ is
		trivial in the sense of Definition 1.8.11 in p.41 of \cite{Lon4}.
		Consequently the matrix $\xi_{\aa_{k}(e),e}(2\pi)$ is spectrally stable and linearly unstable.

		\item If $\aa_k(e) = \aa_s(e) < \aa_m(e)$,
		we have either $\xi_{\aa_{k}(e),e}(2\pi)  \approx N_1(-1,1)  \diamond D(\lm)$ for some $-1\neq \lm <0$
		and is linearly unstable; or $\xi_{\aa_{k}(e),e}(2\pi) \approx M_2(-1,c)$ with $c_1$, $c_2\in \R$ and $c_2 \neq 0$. Thus it is spectrally stable and linearly unstable.

		\item If $\aa_k(e) = \aa_s(e) = \aa_m(e)$, either $\xi_{\aa_{k}(e),e}(2\pi) \approx M_2(-1,c)$
		with $c_1$, $c_2\in \R$ with $c_2\neq 0$,
		or
		$\xi_{\aa_{k}(e),e} \approx N_1(-1, 1)\diamond N_1(-1, 1)$.
		Thus $\xi_{\aa_{k}(e),e}(2\pi)$ is spectrally stable and linearly unstable.
	\end{enumerate}
\end{theorem}

\begin{proof}
(i) Let $e\in [0,1)$ satisfy $\aa_s(e) < \aa_m(e)$.
Then Corollary \ref{coro:null.sep} implies $\nu_{-1}(\xi_{\aa_{m}(e),e}(2\pi)) =1$.
As the limit case of (i) and (ii) of Theorem \ref{thm:RE.norm.form},
we have the eigenvalues of matrix $\xi_{\aa_{m}(e),e}(2\pi)$ are all on $\U$ and the normal form is either $M\approx M_2(-1,c)$ for some $c_2 \neq 0$ or $M \approx N_1(-1,1)\diamond R(\th)$ for some $\th \in (\pi,2\pi)$.

By Lemma \ref{lem:boundary} and $\aa_s(e) < \aa_m(e)$, we have that  $M\approx M_2(-1,c)$ for some $c_2 \neq 0$ cannot holds.
The normal form is given by $M \approx N_1(-1,1)\diamond R(\th)$.
So $M$ is spectrally stable and linearly unstable.

(ii) Let $e\in [0,1)$  satisfy $\aa_k(e)<\aa_s(e) = \aa_m(e)$.
As the limit case of (i) and (ii) of the Theorem \ref{thm:RE.norm.form} and Lemma \ref{lem:boundary},
the normal form of the matrix $M \equiv \xi_{\aa_s(e),e}(2\pi)$ is either $M \approx N_1(-1,a) \diamond N_1(-1,b)$ for some $a$, $b \in \{-1,1\}$,
or $M \approx -I_2\diamond R(\th)$ for some $\th \in (\pi,2\pi)$.
However, the case of $M \approx N_1(-1,a) \diamond N_1(-1,b)$ is impossible by Lemma \ref{lem:boundary}.
Then we have that  $M \approx -I_2 \diamond R(\th)$ for some $\th \in (\pi,2\pi)$ and it is linear stable but not strongly linear stable.

(iii) Let $e\in [0,1)$ satisfy $\aa_k(e)<\aa_s(e) < \aa_m(e)$.
As the limiting case of Cases (ii) and (iii) of Theorem  \ref{thm:RE.norm.form}, the normal form of the matrix $M \equiv \xi_{ \aa_s(e),e}(2\pi)$
must satisfy either $M \approx N_1(-1,-1)\diamond R(\th)$ for some $\th\in (\pi,2\pi)$ or $M \approx M_2(-1,c)$ for some $c_2\neq 0$.
Here the second case is also impossible by Lemma \ref{lem:boundary}, and the conclusion of (iii) follows.

(iv)
Let $e\in [0,1)$ satisfy $\aa_k(e)<\aa_s(e) \leq \aa_m(e)$.
As the limiting case of the cases (iii) of Theorem \ref{thm:RE.norm.form},
the matrix $\xi_{\aa_k(e),e}(2\pi)$ must have Krein collision eigenvalues $\sg(M) = \{ \lm_1,\bar{\lm}_1, \lm_2,\bar{\lm}_2\}$ with $\lm_1 = \bar{\lm}_2 = e^{\sqrt{-1}\th}$ for some $\th \in (0,\pi) \cup (\pi,2\pi)$.
By Theorem \ref{prop:index} and the definition of $\aa_k(e)$,
$\pm 1$ cannot be the eigenvalue of $\xi_{\aa_k(e),e}(2\pi)$. Therefore, we must have that $M \approx N_2(\om,b)$ for $\om = e^{\sqrt{-1}\th}$ and some matrix
$b=
(\begin{smallmatrix}
b_1 & b_2\\ b_3 & b_4
\end{smallmatrix})$
which has the form of (25-27) by Theorem 1.6.11 in p. 34 of \cite{Lon4}.
Because $(I_2 \diamond -I_2)^{-1} N_2( e^{\sqrt{-1}\th},b)(I_2 \diamond -I_2) = N_2( e^{\sqrt{-1}(2\pi-\th)},\hat{b}) $ where
$\hat{b}=(
\begin{smallmatrix}
b_1 & -b_2\\ -b_3 & b_4
\end{smallmatrix})$.
We can always suppose $\th \in (0,\pi)$ without changing the fact  $M \approx N_2(\om,b)$.

Note that by \eqref{eqn:def.aak} and Lemma \ref{lem:mono.ind.nul}, we have $i_{\om} (\xi_{\aa_{k},e}) = 0$.
Suppose $b_2-b_3 =0$, by Lemma 1.9.2 in p. 43 of \cite{Lon},
$\nu_{\om} (N_2(\om,b)) =2$ and then $N_2(\om,b)$ has basic normal form $R(\th)\diamond R(2\pi-\th)$ by the study in case 4
in p. 40 of \cite{Lon4}.
Thus we have following contradiction.
\begin{align}
  0=i_{\omega}\left(\xi_{\aa_{k}(e), e}\right)=i_{1}\left(\xi_{\aa_{k}(e), e}\right)+S_{M}^{+}(1)-S_{R(\theta)}^{-}(\omega)-S_{R(\theta)}^{-}(\bar{\omega}) \leq-1.
\end{align}
  Therefore $b_2 -b_3 \neq 0$ must hold. Then we obtain
  \begin{align}
  0=i_{\omega}(\xi_{\aa,\bb_{k}(e), e})=i_{1}(\xi_{\aa,\bb_{k}(e), e})+S_{M}^{+}(1)-S_{N_{2}(\omega, b)}^{-}(\omega)=-S_{N_{2}(\omega, b)}^{-}(\omega).
\end{align}
By $\<6\>$ and $\<7\>$ in the list of splitting number in Section \ref{sec:Prelim}, we obtain that $N_2(\om, b)$ must be
trivial.
Then by Theorem 1 of \cite{ZhuLong}, the matrix $M$ is spectrally stable and is linearly unstable as claimed.

(v) Let $e\in [0,1)$ satisfy $\aa_k(e) =\aa_s(e) < \aa_m(e)$. Note that $-1$ must be an eigenvalue of $M \approx \xi_{\aa_k(e),e}(2\pi)$ with the geometric multiplicity 1 by Proposition \ref{prop:index}.
As the limit case of (ii) and (iii) of Theorem \ref{thm:RE.norm.form}, the matrix $M$ must satisfy either $M \approx M_2(-1,b)$ with $b_1$,
$b_2 \in \R$ and $b_2 \neq 0$, and thus is spectrally stable and linearly unstable;
or $M \approx N_{-1}(-1,a)\diamond D(\lm)$ for some $a\in \{-1,1\}$ and $-1\neq \lm < 0$. Then in the later case we obtain
\begin{align}
  0=i_{-1}(\xi_{\aa_{k}(e), e})=i_{1}(\xi_{\aa_{k}(e), e})+S_{M}^{+}(1)-S_{N_{1}(-1, a)}^{-}(-1)=-S_{N_{1}(-1, a)}^{-}(-1).
\end{align}
Then by $\<3\>$ and $\<4\>$ in the list of splitting number in Section \ref{sec:Prelim}, we must have $a = 1$.
Thus $M = \xi_{\aa_{k}(e), e}(2\pi)$ is hyperbolic (elliptic-hyperbolic) and linearly unstable.
Note that by the above argument, the matrix $M_2(-1, b)$ also has the basic
normal form $N_1(-1, 1)\diamond	D(\lm)$ for some $-1 \neq \lm < 0$.

(vi) Let $e\in [0,1)$ satisfy $\aa_k(e) =\aa_s(e) = \aa_m(e)$.
As the limiting case of
cases (i), (ii) and (iii) of Theorem \ref{thm:RE.norm.form}, $-1$ must be the only eigenvalue of
$M= \xi_{\aa_k(e),e}$ with $\nu_{-1}(M) =2$ by  Proposition \ref{prop:index}. Thus the matrix $M$ must satisfy $M \approx M_1(-1,c)$ with $c =0$ and $v_{-1}(M_2(-1,c)) = 2$; or $M \approx N_1(-1,\hat{a}) \diamond N_1(-1,\hat{b})$ for some $\hat{a}$ and $\hat{b} \in \{-1,1\}$. In both case $M$ possesses the basic normal form $ N_1(-1,\hat{a}) \diamond N_1(-1,\hat{b})$ for some $\hat{a}$ and $\hat{b} \in \{-1,1\}$. Thus we obtain
\begin{align}
  0 &= i_{-1}\left(\xi_{\aa_{k}(e), e}\right) \nn\\ &=i_{1}\left(\xi_{\aa_{k}(\aa,e), e}\right)+S_{M}^{+}(1)-S_{N_{1}(-1, a)}^{-}(-1)-S_{N_{1}(-1, b)}^{-}(-1) \nn\\
  &=-S_{N_{1}(-1, a)}^{-}(-1)-S_{N_{1}(-1, b)}^{-}(-1).
\end{align}
Then by $\<3\>$ and $\<4\>$ in the list of splitting number in Section \ref{sec:Prelim}, we must have $a = b = 1$ similar
to our above study for (v). Therefore it is spectrally stable and linearly unstable as claimed.
\end{proof}

For the circular case, by $e = 0$, we obtain the value of $\aa_s$ and $\aa_m$ and directly by Section 3.3 of \cite{HLS} and $\bb = 9 - \aa^2$.

\begin{corollary}\lb{coro:cir.sta}
  For the circular orbit, namely $e = 0$, we have that $\aa_s = \aa_m = \frac{\sqrt{33}}{2}$ and $\aa = 2\sqrt{2}$. 
  \begin{enumerate}[label = (\roman*)]
    \item If $0 \leq \aa < 2\sqrt{2}$, we have that $\sigma(\xi_{\aa, e}) \subset \C\bs (U\cup \R)$, and thus it is linearly unstable.
    \item If $\aa = 2\sqrt{2}$, we have that $\xi_{\aa, e} \approx N_2(\omega_0, b)$ with $\om = e^{\sqrt{-1}\sqrt{2}\pi}$ and $b = (\begin{smallmatrix}
      b_1 & b_2 \\ b_3 & b_4
    \end{smallmatrix})$ with $b_2 \neq b_3$ and thus it is linearly unstable.
    \item If $2\sqrt{2} <\aa < \frac{\sqrt{33}}{2} $, we have that $\xi_{\aa, e} \approx R\left(\theta_{1}\right) \diamond R\left(\theta_{2}\right)$ for some $\theta_{1}\in (0, \pi)$ and $\theta_{2} \in(\pi, 2 \pi)$, and thus it is linearly stable.

    \item If $\aa =\frac{\sqrt{33}}{2} $, we have that $\xi_{\aa, e} \approx -I_2 \diamond R\left(\sqrt{3}\pi\right)$ and thus it is linearly stable.
    \item If $\frac{\sqrt{33}}{2} < \aa \leq  3$, we have that $\xi_{\aa, e} \approx R\left(\theta_{1}\right) \diamond R\left(\theta_{2}\right)$ for some $\theta_{1}$ and $\theta_{2} \in(\pi, 2 \pi]$, and thus it is linearly stable.
  \end{enumerate}
  
\end{corollary}

\setcounter{equation}{0}
\section{The linear stability of the symmetric case $m_1=m_3$}\label{sec:4}
In this section, we consider the symmetric case. 
Because of the relationship between the linear stability of the restrict 4-body problem and the linear stability of the Lagrangian solutions. 
It is well-known that the linear stability of an elliptic Euler solution
of the $3$-body problem with masses $m=(m_1,m_2,m_3)\in (\R^+)^3$ is determined by the
eccentricity $e\in [0,1)$ and the mass parameter
\begin{align}\lb{1.4}
  \bb = \frac{m_1(3x^2+3x+1)+m_3x^2(x^2+3x+3)}{x^2+m_2[(x+1)^2(x^2+1)-x^2]},
\end{align}
where $x$ is the unique positive solution of the Euler quintic polynomial equation
\be (m_3+m_2)x^5+(3m_3+2m_2)x^4+(3m_3+m_2)x^3-(3m_1+m_2)x^2-(3m_1+2m_2)x-(m_1+m_2)=0, \lb{Euler.quintic.polynomial}
\ee
and the three bodies form a central configuration of $m$, which are denoted by $q_1=0$,
$q_2=(x\alpha,0)^T$ and $q_3=((1+x)\alpha,0)^T$ with $\alpha=|q_2-q_3|>0$, $x\alpha=|q_1-q_2|$ \cite{Zhou2017}.

If further assuming that $m_1 = m_3$, we have that $x=1$ is the unique positive root of \eqref{Euler.quintic.polynomial}.
Since the center of mass of the three primaries is the origin, we have the position $a_i$ can be written in $m_2$ as follows.
\begin{align}
a_1 = (-(1-m_2)^{-{1\over2}},0)^T,\quad 
a_2 = (0,0)^T,\quad 
a_3 = ((1-m_2)^{-{1\over2}},0)^T.
\end{align}
Since the center of mass is the origin,
we suppose $a_4=(0,y(1-m_2)^{-{1\over2}})^T$ for some $y\in\R^+$.
By \eqref{second.formula.of.mu}, we have
\begin{equation}\label{y}
{1-m_2\over(y^2+1)^{3\over2}} + {m_2\over y^3}={1+7m_2\over8}.
\end{equation}
Note that when $m_2=0$, \eqref{y} gives $y=\sqrt{3}$;
when $m_2=1$, \eqref{y} gives $y=1$.
If $0<y<1$, we have
\begin{equation}
{1-m_2\over(y^2+1)^{3\over2}} + {m_2\over y^3}\ge{1-m_2\over\sqrt{8}}+m_2\ge{1-m_2\over8}+m_2={1+7m_2\over8},
\end{equation}
where the equality holds if and only if $m_2=1$.
Hence we must have $y\ge1$.
Moreover, we take the derivative of $y$ and obtain that 
\begin{equation}\label{y.prime}
{\d y\over \d m_2}
=-\frac{{7\over8}+{1\over(y^2+1)^{3\over2}}-{1\over y^3}}
       {3y\Big[{1-m_2\over(y^2+1)^{5\over2}} + {m_2\over y^5}\Big]}
       <0,
\end{equation}
where $m_2\in[0,1)$.
Here we used
$$
{1\over y^3}-{1\over(y^2+1)^{3\over2}}<{1\over y^3}-{1\over (y+1)^3}
={3y^2+3y+1\over y^3(y+1)^3}={3\over y^2(y+1)^2}+{1\over y^3(y+1)^3}
<{3\over4}+{1\over8}={7\over8}.
$$
Therefore, the range of $y$ is $[1,\sqrt{3}]$,
and $y$ is strictly decreasing with respect to $m_2\in[0,1]$.
By direct computations,  we can reduce $D$, defined in \eqref{eqn:D}, to 
\begin{equation}
D=3\left(\begin{matrix}
{8(1-m_2)\over(1+7m_2)(y^2+1)^{5\over2}} & 0\\
0 & 1 - {8(1-m_2)\over(1+7m_2)(y^2+1)^{5\over2}}
\end{matrix}\right).
\end{equation}
Let
\begin{equation}\label{z}
z = {8(1-m_2)\over(1+7m_2)(y^2+1)^{5\over2}}.
\end{equation}
We then have $\lambda_3=3(1-z)$, $\lambda_4=3z$ and $\det D=9z(1-z)$. It follows that 
\begin{equation}\label{alpha}
\alpha = \lambda_3-\lambda_4=6({1\over2}-z).
\end{equation}

If $m_2=0$, then $y=\sqrt{3}$, hence \eqref{z} gives $z={1\over4}$,
and \eqref{alpha} gives $\alpha={3\over2}$;
if $m_2=1$, \eqref{z} gives $z=0$,
and \eqref{alpha} gives $\alpha=3$.
Figure \ref{fig:2} and Corollary \ref{coro:cir.sta} we have the relation of $\alpha$ and $m_2$ and we obtain the stability of the system in term of $m_2$  and Theorem \ref{thm:numerically} holds.

\appendix
\setcounter{equation}{0}
\section{Appendix}
\subsection{Preliminaries of $\om$-Maslov-type indices and $\om$-Morse indices}\lb{sec:Prelim}
Denote by $\Sp(2n)$ the symplectic group of real $2n\times 2n$ matrices.
Following \cite{Lon4}, for any $\omega\in\U=\{z\in\C\;|\;|z|=1\}$, define a real function
$D_\om(M)=(-1)^{n-1}\overline{\om}^n det(M-\om I_{2n})$ for any $M$ in the symplectic group $\Sp(2n)$.
Then we can define $\Sp(2n)_{\om}^0 = \{M\in\Sp(2n)\,|\, D_{\om}(M)=0\}$ and
$\Sp(2n)_{\om}^{\ast} = \Sp(2n)\bs \Sp(2n)_{\om}^0$. The orientation of $\Sp(2n)_{\om}^0$ at any of its point
$M$ is defined to be the positive direction $\frac{\d}{\d t}Me^{t J}|_{t=0}$ of the path $Me^{t J}$ with $t>0$ small
enough. Let $\nu_{\om}(M)=\dim_{\C}\ker_{\C}(M-\om I_{2n})$. Let
$\mathcal{P}_{2\pi}(2n) = \{\ga\in C([0,2\pi],\Sp(2n))\;|\;\ga(0)=I\}$ and
$\ga_0(t)=\diag(2-\frac{t}{2\pi}, (2-\frac{t}{2\pi})^{-1})$ for $0\le t\le 2\pi$.

Given any two $2m_k\times 2m_k$ matrices of square block form
$M_k = (\begin{smallmatrix}
A_k & B_k\\ C_k & D_k
\end{smallmatrix} )$ with $k=1, 2$,
the symplectic sum of $M_1$ and $M_2$ is defined (cf. \cite{Lon4}) by
the following $2(m_1+m_2)\times 2(m_1+m_2)$ matrix $M_1\dm M_2$:
$$
M_1\dm M_2=
\begin{pmatrix}
       A_1 &   0 & B_1 &   0\\
       0   & A_2 &   0 & B_2\\
       C_1 &   0 & D_1 &   0\\
      0   & C_2 &   0 & D_2  
\end{pmatrix}.
$$
 For any two paths $\xi_j\in\P_{\tau}(2n_j)$
 with $j=0$ and $1$, let $\xi_0\dm\xi_1(t)= \xi_0(t)\dm\xi_1(t)$ for all $t\in [0,\tau]$.
As in \cite{Lon4}, for $\lm\in\R\bs\{0\}$, $a\in\R$, $\th\in (0,\pi)\cup (\pi,2\pi)$,
$b=(\begin{smallmatrix}b_1 & b_2\\
                 b_3 & b_4\end{smallmatrix})$ with $b_i\in\R$ for $i=1, \ldots, 4$, and $c_j\in\R$
for $j=1, 2$, some normal forms are given by
\begin{align}
& D(\lm)=\begin{pmatrix}\lm & 0\\
                         0  & \lm^{-1} \end{pmatrix}, \qquad
R(\th)=\begin{pmatrix}\cos\th & -\sin\th\\
                     \sin\th  & \cos\th\end{pmatrix},  \qquad
 N_1(\lm, a)=\begin{pmatrix}\lm & a\\
                             0   & \lm\end{pmatrix},\\ 
 &  N_2(e^{\sqrt{-1}\th},b) = \begin{pmatrix}R(\th) & b\\
                                           0      & R(\th)\end{pmatrix}, \quad 
 M_2(\lm,c)=\begin{pmatrix}
\lm &   1 &       c_1 &         0 \\
0 & \lm &       c_2 & (-\lm)c_2 \\
0 &   0 &  \lm^{-1} &         0 \\
0 &   0 & -\lm^{-2} &  \lm^{-1} \end{pmatrix}. \lb{eqn:m2}\end{align}
Here $N_2(e^{\sqrt{-1}\th},b)$ is {\bf trivial} if $(b_2-b_3)\sin\th>0$, or {\bf non-trivial}
if $(b_2-b_3)\sin\th<0$, in the sense of Definition 1.8.11 on p.41 of \cite{Lon4}. Note that
by Theorem 1.5.1 on pp.24-25 and (1.4.7)-(1.4.8) on p.18 of \cite{Lon4}, when $\lm=-1$ there hold
$c_2 \not= 0$ if and only if $\dim\ker(M_2(-1,c)+I)=1$ and
$c_2 = 0$ if and only if $\dim\ker(M_2(-1,c)+I)=2$.
For more details, readers may refer Section 1.4-1.8 of \cite{Lon4}.

For any symplectic path $\xi\in \mathcal{P}_{2\pi}(2n)$ we define $\nu_\om(\ga)=\nu_\om(\xi(2\pi))$ and
$$  i_\om(\xi)=[\Sp(2n)_\om^0:\xi\ast\xi_n], \qquad {\rm if}\;\;\xi(2\pi)\not\in \Sp(2n)_{\om}^0,  $$
i.e., the usual homotopy intersection number, and the orientation of the joint path $\xi\ast\xi_n$ is
its positive time direction under homotopy with fixed end points, where the path $\xi_n(t) = (\begin{smallmatrix}
2-\frac{t}{\tau} & 0 \\ 0 &  (2-\frac{t}{\tau})^{-1}
\end{smallmatrix})$. When $\xi(2\pi)\in \Sp(2n)_{\om}^0$, the index $i_{\om}(\xi)$  follows \cite{Lon4}. The pair
$(i_{\om}(\xi), \nu_{\om}(\xi)) \in \Z\times \{0,1,\ldots,2n\}$ is called the index function of $\xi$
at $\om$. When $\nu_{\om}(\xi)=0$ or $\nu_{\om}(\xi)>0$, the path $\xi$ is called $\om$-{\it non-degenerate}
or $\om$-{\it degenerate} respectively.
For more details readers may refer to \cite{Lon4}.

For $T>0$, suppose $x$ is a critical point of the functional
$$ F(x)=\int_0^TL(t,x,\dot{x})dt,  \qquad \forall\,\, x\in W^{1,2}(\R/T\Z,\R^n), $$
where $L\in C^2((\R/T\Z)\times \R^{2n},\R)$ and satisfies the
Legendrian convexity condition $L_{p,p}(t,x,p)>0$. It is well known
that $x$ satisfies the corresponding Euler-Lagrangian
equation:
\begin{align}
  & \frac{\d}{\d t}L_p(t,x,\dot{x})-L_x(t,x,\dot{x})=0,    \label{A2.7}\\
  & x(0)=x(T),  \qquad \dot{x}(0)=\dot{x}(T).    \label{A2.8}
\end{align}

For such an extremal loop, define
\begin{align}
  P(t) = L_{p,p}(t,x(t),\dot{x}(t)),\quad
  Q(t) = L_{x,p}(t,x(t),\dot{x}(t)),\quad
  R(t) = L_{x,x}(t,x(t),\dot{x}(t)).
\end{align}
Note that $F\,''(x)=-\frac{\d}{\d t}(P\frac{\d}{\d t}+Q)+Q^T\frac{\d}{\d t}+R$.

For $\omega\in\U$, set
$
  D(\omega,T)=\{y\in W^{1,2}([0,T],\C^n)\,|\, y(T)=\omega y(0) \}.  
$
We define the $\omega$-Morse index $\phi_\omega(x)$ of $x$ to be the dimension of the
largest negative definite subspace of $ \langle F\,''(x)y_1,y_2 \rangle$, for all $y_1,y_2\in D(\omega,T)$,
where $\langle\cdot,\cdot\rangle$ is the inner product in $L^2$. For $\omega\in\U$, we
also set
\begin{align}
  \ol{D}(\omega,T)= \{y\in W^{2,2}([0,T],\C^n)\,|\, y(T)=\omega y(0), \dot{y}(T)=\om\dot{y}(0) \}.
                      \label{A.4}
\end{align}
Then $F''(x)$ is a self-adjoint operator on $L^2([0,T],\R^n)$ with domain $\ol{D}(\omega,T)$.
We also define the $\om$-nullity $\nu_{\om}(x)$ of $x$ by
$$ \nu_{\om}(x)=\dim\ker(F''(x)).  $$
Note that we only use $n=2$ in \eqref{A.4} from in this paper.

In general, for a self-adjoint linear operator $A$ on the Hilbert space $\mathscr{H}$,
we set $\nu(A)=\dim\ker(A)$ and denote by $\phi(A)$ its Morse index which is the maximum dimension
of the negative definite subspace of the symmetric form $\langle A\cdot,\cdot\rangle$. Note
that the Morse index of $A$  is equal to the total multiplicity of the negative eigenvalues
of $A$.

On the other hand, $\td{x}(t)=(\partial L/\partial\dot{x}(t),x(t))^T$ is the solution of the
corresponding Hamiltonian system of \eqref{A2.7}-\eqref{A2.8}, and its fundamental solution
$\gamma(t)$ is given by
\begin{align}
  \begin{cases}
    \dot{\gamma}(t) &= JB(t)\gamma(t), \\
       \gamma(0) &= I_{2n},
  \end{cases}
\end{align}
with
$B(t)=\left(\begin{smallmatrix} P^{-1}(t)& -P^{-1}(t)Q(t)\\
                         -Q(t)^TP^{-1}(t)& Q(t)^TP^{-1}(t)Q(t)-R(t) \end{smallmatrix}\right)$.

\begin{lemma} \lb{lem:2.1}(\cite{Lon4}, p.172)\lb{L2.3}
For the $\omega$-Morse index $\phi_\omega(x)$ and nullity $\nu_\omega(x)$ of the solution $x=x(t)$
and the $\omega$-Maslov-type index $i_\omega(\gamma)$ and nullity $\nu_\omega(\gamma)$ of the
symplectic path $\ga$ corresponding to $x$, for any $\omega\in\U$ we have $ \phi_\omega(x) = i_\omega(\gamma)$, and $\nu_\omega(x) = \nu_\omega(\gamma). $
\end{lemma}

A generalization of the above lemma to arbitrary  boundary conditions is given in \cite{HS1}.
For more information on these topics, we refer to \cite{Lon4}.
To measure the jumps between $i_\om(\ga)$
and $i_\lambda(\ga)$ with $\lambda\in\U$ near $\om$ from two sides of $\om$ in $\U$, the splitting numbers $S_M^{\pm}(\om)$ is defined by followings.
\begin{definition} (\cite{Lon4})\lb{D2.3}
For any $M\in\Sp(2n)$ and $\om\in\U$, choosing $\tau>0$ and $\ga\in\P_\tau(2n)$ with $\ga(\tau)=M$,
we define
\begin{align}
  S_M^{\pm}(\om)=\lim_{\epsilon\rightarrow0^+}\;i_{\exp(\pm\epsilon\sqrt{-1})\om}(\ga)-i_\om(\ga).
\end{align}
They are called the splitting numbers of $M$ at $\om$.
\end{definition}

For any $\om_0=e^{\sqrt{-1}\th_0}\in\U$ with $0\le\th_0<2\pi$,
the eigenvalues of $M$ on $\U$ are denote by $\om_j$ with $1\le j\le p_0$ which are
distributed counterclockwise from $1$ to $\om_0$ and located strictly between $1$ and $\om_0$.
Then we have
\begin{align}
  i_{\om_0}(\ga)=i_1(\ga)+S_M^+(1)+\sum_{j=1}^{p_0}(-S_M^-(\om_j)+S_M^+(\om_j))-S_M^-(\om_0).  \lb{eqn:split}
\end{align}
The splitting numbers have following properties.
\begin{lemma} (\cite{Lon4}, p.198)
The integer valued splitting number pair $(S_M^+(\om),S_M^-(\om))$ defined for all
$(\om,M)\in\U\times\cup_{n\ge1}\Sp(2n)$ are uniquely determined by the following axioms:

$1^{\circ}$ (Homotopy invariant) $S_M^{\pm}(\om)=S_N^{\pm}(\om)$ for all $N\in\Omega^0(M)$.

$2^{\circ}$ (Symplectic additivity) $S_{M_1\diamond M_2}^{\pm}(\om)=S_{M_1}^{\pm}(\om)+S_{M_2}^{\pm}(\om)$
for all $M_i\in\Sp(2n_i)$ with $i=1,2$.

$3^{\circ}$ (Vanishing) $S_M^{\pm}(\om)=0$ if $\om\not\in\sigma(M)$.

$4^{\circ}$ (Normality) $(S_M^+(\om),S_M^-(\om))$ coincides with the ultimate type of
$\om$ for $M$ when $M$ is any basic normal form.
\end{lemma}

Moreover, by Lemma 9.1.6 on p.192 of \cite{Lon4} for $\om\in\C$ and $M\in\Sp(2n)$, we have
\begin{align}
  S_M^+(\om)=S_M^-(\overline\om).
\end{align}

The ultimate type of $\om\in\U$ for a symplectic matrix $M$ mentioned in the above lemma
is given in Definition 1.8.12 on pp.41-42 of \cite{Lon4} algebraically with its more properties studied
there.

For the reader's conveniences, following the List 9.1.12 on pp.198-199 of \cite{Lon4},
the splitting numbers (i.e., the ultimate types) for all basic normal forms are given by:

$\langle$1$\rangle$ $(S_M^+(1),S_M^-(1))=(1,1)$ for $M=N_1(1,b)$ with $b=1$ or $0$.

$\langle$2$\rangle$ $(S_M^+(1),S_M^-(1))=(0,0)$ for $M=N_1(1,-1)$.

$\langle$3$\rangle$ $(S_M^+(-1),S_M^-(-1))=(1,1)$ for $M=N_1(-1,b)$ with $b=-1$ or $0$.

$\langle$4$\rangle$ $(S_M^+(-1),S_M^-(-1))=(0,0)$ for $M=N_1(-1,1)$.

$\langle$5$\rangle$ $(S_M^+(e^{\sqrt{-1}\th}),S_M^-(e^{\sqrt{-1}\th}))=(0,1)$
for $M=R(\theta)$ with $\th\in(0,\pi)\cup(\pi,2\pi)$.

$\langle$6$\rangle$ $(S_M^+(\om),S_M^-(\om))=(1,1)$ for $M=N_2(\om,b)$
being non-trivial (cf. Definition 1.8.11 on p.41 of \cite{Lon4}) with
$\om=e^{\sqrt{-1}\th}\in\U\backslash\R$.

$\langle$7$\rangle$ $(S_M^+(\om),S_M^-(\om))=(0,0)$ for $M=N_2(\om,b)$
being trivial (cf. Definition 1.8.11 on p.41 of \cite{Lon4}) with
$\om=e^{\sqrt{-1}\th}\in\U\backslash\R$.

$\langle$8$\rangle$ $(S_M^+(\om),S_M^-(\om))=(0,0)$ for $\om\in\U$ and $M\in\Sp(2n)$
satisfying $\sigma(M)\cap\U=\emptyset$.

For any symplectic path $\ga \in \mathcal{P}_{2\pi}(2n)$ and $m \in \N$, the  $m$-th iteration $\ga_m:[0,\tau] \longrightarrow \Sp(2n)$ is given by
$\ga^m: [0,m\tau] \longrightarrow \Sp(2n),$
with $\ga^m(t) = \ga(t-j\tau)\ga(\tau)^j$ for $g\tau\le t \leq (j+1)\tau$ and $j\in\{0,1,\dots, m-1\}$.
The next Bott-type iteration formula is will be used in this paper.
\begin{lemma}[cf. Theorem 9.2.1 of \cite{Lon4}]\lb{lem:Bott}
For any $z\in \U$,
$i_z(\ga^m) = \sum_{\om^m=z}i_{\om}(\ga).$
\end{lemma}


\begin{thebibliography}{}

\bibitem{Euler} L. Euler, De motu rectilineo trium corporum se mutuo
attrahentium. {\it Novi Comm. Acad. Sci. Imp. Petrop.}  11. (1767) 144-151.

\bibitem{Ga} G. Gascheau,  Examen d'une classe d'\'{e}quations diff\'{e}rentielles
et application \`{a} un cas particulier du probl\`{e}me des trois corps. {\it Comptes
Rend. Acad. Sciences.} 16. (1843) 393-394.

\bibitem{HLS} X. Hu, Y. Long, S. Sun, Linear stability of elliptic Lagrangian
solutions of the classical planar three-body problem via index theory.
{\it Arch. Ration. Mech. Anal.} 213. (2014) 993-1045.

\bibitem{Hu2020}
X. Hu, Y. Long, and Y. Ou.
\newblock Linear stability of the elliptic relative equilibrium with (1+n)-gon
central configurations in planar n-body problem.
\newblock {\em Nonlinearity}, 33(3):1016--1045, {2020}.

\bibitem{HS} X. Hu, S. Sun, Morse index and stability of elliptic Lagrangian
solutions in the planar three-body problem. {\it Adv. Math.} 223. (2010) 98-119.


\bibitem{HuOuWang2015ARMA}
X. Hu, Y. Ou, and P. Wang.
\newblock Trace formula for linear {H}amiltonian systems with its applications
to elliptic {L}agrangian solutions.
\newblock {\em Arch. Ration. Mech. Anal.}, 216(1):313--357, 2015.

\bibitem{IM} R. Iturriaga, E. Maderna, Generic uniqueness of the minimal
Moulton central configuration. {\it Celest. Mech. Dyn. Astr.} 123 (2015), 351-361.


\bibitem{Lag} J. L. Lagrange,  Essai sur le probl\`{e}me des trois corps. Chapitre II.
{\OE}uvres Tome {6}, Gauthier-Villars, Paris. (1772) 272-292.

\bibitem{Leandro2003}
E. S.~G. Leandro.
\newblock Finiteness and bifurcations of some symmetrical classes of central
configurations.
\newblock {\em   {Arch. Ration. Mech. Anal.}}, 167(2):147--177,
{2003}.

\bibitem{Leandro2018}
E. S.~G. Leandro.
\newblock Structure and stability of the rhombus family of relative equilibria
under general homogeneous forces.
\newblock {\em  {J. Dynam. Differential Equations}}, 31(2):933--958,  {2018}.


\bibitem{Lio} J. Liouville, Sur un cas particulier du probl\`{e}me des trois corps.
{\it J. Math. Pures Appl.} 7. (1842) 110-113.

\bibitem{Liu2021}
B. Liu.
\newblock Linear instability of elliptic rhombus solutions to the planar four-body problem.
\newblock {\em Nonlinearity}, Nonlinearity 34 (11), 7728–7749, 2021.


\bibitem{LiZ} B. Liu, Q. Zhou, Linear stability of elliptic relative equilibria of restricted four-body problem.
{\it J. Diff. Equa.} 269. (2020) 4751-4798.

\bibitem{Lon4} Y. Long, Index Theory for Symplectic Paths with
	Applications. Progress in Math. 207, Birkh\"auser. Basel. 2002.

\bibitem{Lon} Y. Long, Lectures on Celestial Mechanics and Variational Methods.
{\it Preprint.} 2012.

\bibitem{Mansur2017}
A. Mansur, D. Offin, and M. Lewis.
\newblock Instability for a family of homographic periodic solutions in the
parallelogram four body problem.
\newblock {\em  {Qual. Theory Dyn. Syst.}}, 16(3):671--688,  {2017}.

\bibitem{MSS} R. Mart\'{\i}nez, A. Sam\`{a}, C. Sim\'{o},
Stability of homograpgic solutions of the planar three-body problem
with homogeneous potentials. in International conference on
Differential equations. Hasselt, 2003, eds, Dumortier, Broer, Mawhin,
Vanderbauwhede and Lunel, World Scientific, (2004) 1005-1010.

\bibitem{MSS1} R. Mart\'{\i}nez, A. Sam\`{a}, C. Sim\'{o},
Stability diagram for 4D linear periodic systems with applications
to homographic solutions. {\it J. Diff. Equa.} 226. (2006) 619-651.

\bibitem{MSS2} R. Mart\'{\i}nez, A. Sam\`{a}, C. Sim\'{o}, Analysis of
the stability of a family of singular-limit linear periodic systems in
$\R^4$. Applications. {\it J. Diff. Equa.} 226. (2006) 652-686.


\bibitem{MS} K. Meyer, D. Schmidt, Elliptic relative equilibria in
the N-body problem. {\it J. Diff. Equa.} 214. (2005) 256-298.

\bibitem{Mou} F. Moulton, The straight line solutions of the $n$-body problem.
{\it Ann. of Math}. II Ser. 12 (1910) 1-17.

\bibitem{R1} G. Roberts, Linear stability of the elliptic Lagrangian
triangle solutions in the three-body problem. {\it J. Diff. Equa.}
182. (2002) 191-218.

\bibitem{R2} E. Routh, On Laplace's three particles with
a supplement on the stability or their motion. {\it Proc. London Math.
Soc.} 6. (1875) 86-97.


\bibitem{zhou2019linear}
Q. Zhou.
\newblock Linear stability of elliptic relative equilibria of four-body problem
with two infinitesimal masses.
\newblock {\em arXiv preprint arXiv:1908.01345}, 2019.

\bibitem{Zhou2017}
Q. Zhou and Y. Long.
\newblock Maslov-type indices and linear stability of elliptic {E}uler solutions
of the three-body problem.
\newblock {\em  {Arch. Ration. Mech. Anal.}}, 226(3):1249--1301,
{2017}.

\bibitem{ZhonLong2017CMDA}
Q. Zhou and Y. Long.
\newblock The reduction of the linear stability of elliptic {E}uler-{M}oulton
solutions of the {$n$}-body problem to those of 3-body problems.
\newblock {\em {Celestial Mech. Dynam. Astronom.}}, 127(4):397--428, 2017.

\bibitem{ZhuLong} G. Zhu, Y. Long
Linear stability of some symplectic matrices.
{\it Front. Math. China.} 5(2010), no. 2, 361–368.

\end{thebibliography}
\end{document}